\documentclass[12pt]{article}
\usepackage{latexsym}
\usepackage{amssymb,amsfonts,enumerate}
\usepackage{amsmath}
\newtheorem{theorem}{Theorem}

\newtheorem{lemma}{Lemma}

\newtheorem{prop}{Proposition}

\newtheorem{remark}{Remark}

\newenvironment{proof}{\medskip \noindent
{\bf Proof.}}{\hfill \rule{.5em}{1em}
\\}

\def\bea{\begin{eqnarray*}}
\def\eea{\end{eqnarray*}}
\def\be{\begin{equation}}
\def\ee{\end{equation}}

\begin{document}

\title{Classification of Gradient Ricci solitons  with harmonic Weyl curvature}

\author{ Jongsu Kim
\thanks{This research was supported by Basic Science Research Program through the National Research Foundation of Korea (NRF) funded by the Ministry of Science, ICT and Future Planning (NRF-2020R1A2B5B01001862).
} }

%\author{ Jongsu Kim}
%\date{today}

%\address{Dept. of Mathematics, Sogang University, Seoul, Korea}
%\email{jskim@sogang.ac.kr}

%\thanks{This research was supported by Basic Science Research Program through the National Research Foundation of Korea (NRF) funded by the Ministry of Science, ICT and Future Planning (NRF-2020R1A2B5B01001862)}

%\keywords{gradient Ricci soliton,  harmonic Weyl curvature}

%\subjclass{53C21, 53C25}
%\subjclass[2010]{53C21, 53C25}

\maketitle

\begin{abstract}
We make classifications  of gradient Ricci solitons $(M, g, f)$ with harmonic Weyl curvature.   As a local classification, we  prove that the soliton metric $g$  is locally isometric to  one of the following four types: an Einstein manifold, the Riemannian product of a  Ricci flat manifold and an Einstein manifold, a warped product of $\mathbb{R}$ and an Einstein manifold,  and a singular warped product of $\mathbb{R}^2$ and a Ricci flat manifold.
Compared with the previous four-dimensional study in \cite{Ki},   we have developed a novel method of {\it refined adapted frame fields} and overcome the main difficulty arising from a large number of Riemmannian connection components in  dimension$ \geq 5$.

 Next we have obtained a classification of {\it complete} gradient Ricci solitons with harmonic Weyl curvature. For the proof, using the real analytic nature of $g$ and $f$,  we elaborate geometric arguments to fit together local regions.
\end{abstract}

%\theoremstyle{plain}
%\newtheorem{theorem}{Theorem}[section]

%\numberwithin{equation}{section}
%\documentstyle{amsart}--amstex language
%\newtheorem{thm}{Theorem}

\newtheorem{conj}[theorem]{Conjecture}

\section{Introduction}

 A gradient Ricci soliton (henceforth abbreviated as GRS) is a Riemannian manifold $(M, g)$ together with a smooth function $f$ satisfying
\begin{equation} \label{grs}
\nabla d f =  -r + \lambda g,
\end{equation}
where $r$ is the Ricci tensor of $g$ and $\lambda$  a constant.
  GRS is essential in Hamilton's Ricci flow theory as
singularity models of the flow. They have been studied extensively for decades from diverse points of view.

There are a number of literatures
concerned with  geometric characterization of gradient Ricci solitons.  For two or three dimensional GRS,   one may refer to
 \cite{BM,  Br, Cao1,  Iv3, NW, P} and references therein.
For higher dimension,  several hypotheses on the Weyl curvature tensor $W$ have been studied.
Indeed, complete shrinking (i.e. $\lambda>0$) GRS with $W=0$ are classified in \cite{Z} to be a finite quotient of $\mathbb{R}^n$, $\mathbb{S}^n$, or $\mathbb{S}^{n-1} \times \mathbb{R}$, $n \geq  4$
; see also \cite{CWZ,ELM, NW, PW2}. Note that in dimension$\geq 4$, a Riemannian manifold has   $W=0$ if and only if  it is  locally conformally flat.
A complete locally conformally flat steady  (i.e. $\lambda=0$) GRS is  shown to be either flat or isometric to the Bryant soliton \cite{CC1, CM}.
The four dimensional half conformally flat (i.e. $W^{+} =0$ or $W^{-} =0$) GRS with $\lambda \geq 0$ are studied in \cite{CW}, where  $W^{\pm}$ is the positive or negative part of the Weyl curvature tensor $W$.
Bach-flat GRS with $\lambda \geq 0$  are studied in \cite{CCCMM, CC2}. In \cite{WWW}, any four dimensional  shrinking GRS with   $\delta W^{+}=0$ is shown to be rigid; here a GRS is said to be rigid if it is isometric to a
quotient of the Riemannian product $N \times \mathbb{R}^k$ where $N$ is an Einstein manifold and $f = \frac{\lambda}{2}
|x|^2$ on the Euclidean factor.

We are here particularly interested in  GRS with harmonic Weyl curvature.
 Fern\'{a}ndez-L\'{o}pez and Garc\'{i}a-R\'{i}o  \cite{FG} showed that a compact GRS  with harmonic Weyl curvature is rigid.  Munteanu and Sesum \cite{MS} proved that a complete shrinking GRS with
harmonic Weyl tensor is rigid.

The author has classified four dimensional GRS with harmonic Weyl curvature in \cite{Ki}, where
such GRS  are  characterized locally via
 detailed eigenvalue analysis  of the Codazzi tensor $r  - \frac{R}{6} g$, where $R$ is the scalar curvature of $g$. While this method  provided explicit local classifications in three and four dimension \cite{Sh2}, it
   appeared to be ineffective  for  higher dimension,  owing to computational difficulty.

In this article we have overcome the main technical problem in  \cite{Ki}  by developing a new technique of  {\it  refined adapted}  Ricci-eigen frame fields and succeeded in
 classifying GRS with harmonic Weyl curvature without any dimensional restriction.
The first result of this paper is to obtain a complete description of {\it local} soliton metrics and potential functions;

\begin{theorem} \label{locals}
Let $(M^n,  g, f)$, $n \geq 4$, be a (not necessarily complete) $n$-dimensional gradient Ricci soliton with harmonic Weyl curvature.
Then for each point in some open dense subset $\mathfrak{M}$ of $M$, there exists a neighborhood $V$ such that  $(V, g)$ is isometric to a region in  one of the following:

\medskip
{\rm (i)}  an Einstein manifold with $f$ a constant function.

\smallskip
{\rm (ii)}  the Riemannian product of a $k$-dimensional manifold  $(N_1^{k}, g_1) $ with Ricci curvature $r_{g_1} = 0$, $2 \leq k \leq n-2$, and an Einstein manifold $(N_2^{n-k}, g_2)$
with $r_{g_2} =  \lambda g_2 $, $\lambda \neq 0$. The restriction $f|_{V}$ is the pull-back of a function $\hat{f}$ on $N^k_1$ and
$ (N_1^{k}, g_1, \hat{f} )$  is itself a gradient Ricci soliton.
For  a local function $s$  with $d s = \frac{df}{|d f|}$,  $f|_V= \lambda \frac{s^2}{2} $ modulo a constant.

\smallskip
{\rm (iii)} $\mathbb{R} \times W^{n-1}$ with the warped product metric $ ds^2 +    h(s)^2 \tilde{g},$
where $\tilde{g}$ is an Einstein metric on a manifold $W^{n-1}$, and $f=f(s)$ is a non-constant function of $s$ only.

\smallskip
{\rm (iv)}  $  \mathbb{R}^2 \times N^{n-2}$, $n \neq  5$, with
$
g= dx_1^2   + x_1^{\frac{2({n-3})}{{n-1}}} dx_2^2+  x_1^{\frac{4}{{n-1}}} \tilde{g},
$
 where $(x_1, x_2)$ is the standard coordinates on $\mathbb{R}^2$ and ${\tilde{g}}$ is a Ricci flat metric on a manifold $N^{n-2}$.
 Moreover, $\lambda=0$ and $f|_{V}=\frac{2(n-3)}{{n-1}} \ln |x_1| $ modulo a constant.
\end{theorem}
With
the above local classification in Theorem \ref{locals}, we can go on to characterize {\it complete}  GRS with harmonic Weyl curvature.  Among the four types of regions in Theorem 1, a particular treatment is needed for type-(iii) regions.
 In the four dimensional work \cite{Ki}, a type-(iii) region is locally conformally flat, so at least for shrinking and steady case, we could  apply the afore-mentioned classification results  of complete locally conformally flat GRS.  But these results require  nonnegativity of some  curvature which is provided for $W=0$ case  by some pinching estimate of flow theory \cite{Z}, but not yet for $\delta W=0$.

 So, to characterize  complete  GRS $(M, g, f)$ with harmonic Weyl curvature, we instead try to fit together  regions  supplied by  Theorem \ref{locals}.  We carry out some elaborate process in exploiting the real analytic nature of $g$ and $f$, and obtain a fairly good description of the topology and geometry of $(M, g, f)$,  so that finally we manage to reduce it to locally conformally flat case. We obtain the following  classification of  steady GRS.
\begin{theorem} \label{steady}
An $n$-dimensional complete steady gradient Ricci soliton  with harmonic Weyl curvature
is either Ricci flat, or isometric
to the Bryant soliton.
\end{theorem}

Following the proof of  Theorem \ref{steady}, one can reprove Munteanu-Sesum's classification of complete shrinking GRS with harmonic Weyl curvature.
For expanding solitons there are  a number of recent works, e.g. \cite{ CD, Cho, PW, SS} and references therein. We prove;

\begin{theorem} \label{expand}
For an  $n$-dimensional complete expanding gradient Ricci soliton $(M^n,  g, f)$ with harmonic Weyl curvature,   one of the following {\rm (i)}-{\rm (iv)} holds.

\smallskip
{\rm (i)} $(M, g)$ is an Einstein manifold with $f$ a constant function.

\smallskip
{\rm (ii)}
 $(M, g)$ is isometric to a finite quotient of the Riemannian product $\mathbb{R}^k \times  (N^{n-k}, g_2) $  where  $2 \leq k \leq n-2$ and  $ g_2$ is Einstein  with $r_{g_2} =  \lambda g_2 $, $\lambda \neq 0$. Moreover,
 $f = \frac{\lambda}{2} |x|^2$ on the Euclidean factor.

\smallskip
{\rm (iii)}  $(M, g)$ is isometric to $\mathbb{R}^n$ with the warped product metric $ dr^2 +    h(r)^2 \tilde{g}$,  where $r$ is the radius function of the polar ccordinates on $\mathbb{R}^n$, $h(0) =0$  and $\tilde{g}$ on $\mathbb{S}^{n-1}$ has positive constant curvature.

\smallskip
{\rm (iv)}    $(M, g)$ is isometric to a quotient of  $\mathbb{R} \times W$  with the warped product metric $ ds^2 +    h(s)^2 \tilde{g}$,  where $h>0$ on $\mathbb{R}$ and  $\tilde{g}$ is an Einstein metric on an $(n-1)$-dimensional manifold $W$.

\end{theorem}

\medskip

To prove theorem \ref{locals}, we start with
a useful local form of $g$ in the formula (\ref{solba76})  provided by \cite{CC2}.
 We consider an {\it adapted} frame field, which is a local orthonormal  Ricci-eigen frame field $E_i$, $i =1, \cdots, n $ such that $E_1 = \frac{ \nabla f }{  | \nabla f | }$. The corresponding Ricci eigenvalues $\lambda_i$ are shown to be constant on each connected component of $f^{-1}(c)$ for a regular value $c$ of $f$.
From harmonic Weyl curvature condition,  the  Codazzi tensor $r- \frac{R}{2(n-1)}g$ provides
other important  properties on  Ricci eigenvalues and  Ricci eigenspaces
 \cite{De}. With these we get many relations on $\lambda_i$
and the Christoffel symbols $\Gamma_{ij}^k:= \langle  \nabla_{E_i}E_j,   E_k \rangle$ of $g$. In the previous four dimensional work \cite{Ki},
 the number of  $\Gamma_{ij}^k$ is not large so that
we could make a case-by-case study to characterize $g$.
 But in  higher dimension, the number of  $\Gamma_{ij}^k$ increases much, so there remained a nontrivial problem in extending the argument to higher dimension.

\smallskip
The main novelty in this article is to refine an adapted frame field $\{ E_i \}$. In fact,  by translating $E_i$, $i \geq 2$, along the $E_1$ flow, we obtain a refined adapted frame field $\{ F_i\}$ in which one can improve the control on  $\Gamma_{ij}^k$ significantly. By our detailed analysis on the equation (\ref{grs}),
this improved information on  $\Gamma_{ij}^k$ turns out to be crucial
in excluding the existence of  more than two distinct $\lambda_i$ among $i \geq 2$ at each point.
If there are exactly one or two distinct $\lambda_i$ among $i \geq 2$, we do an independent argument to characterize them, and obtain Theorem \ref{locals}.

Then a complete GRS $(M, g, f)$ with harmonic Weyl curvature may be considered as  the union of regions of type (i)-(iv) of Theorem \ref{locals}, matched together along their boundaries. The key point is to understand  how  type-(iii) regions meet one another.
We first see that each type-(iii) region $U$ bordered by the set $\{p \ | \ \nabla f (p) =0 \}$  is  a warped product manifold or its quotient.
Next we show that either   the closure $\overline{U}$ in $M$ is a smooth manifold (with boundary), or some  boundary component(s) of $U$ closes up to one point.
Now by the real analyticity  of $g$ and $f$,  two neighboring such regions can meet nicely along the boundary component. Then $M$ is simple as a smooth manifold, equipped with a global warped product metric in case $M$ has a type-(iii) region.
This gives Theorem \ref{expand}. We argue further to use the classification result of  locally conformally flat GRS to exclude most non-locally-conformally-flat manifolds and obtain Theorem \ref{steady}.

\medskip
This paper is organized as follows. In section 2 we explain properties common to any GRS with harmonic Weyl curvature.
In section 3 some computations are done for a GRS with harmonic Weyl curvature in an adapted Ricci-eigen frame field $\{E_i \} $.
In section 4 we improve  $\{E_i \} $ and get a refined adapted frame field $\{F_i\}$.
In section  5 we compute $\Gamma_{ij}^k$
in $F_i$.
In section  6  we analyze the Ricci tensor  and  GRS equation when there are at least two distinct Ricci-eigen values $\lambda_i$  among $i \geq 2$.
In section  7 we show that a GRS with harmonic Weyl curvature cannot admit  at least three distinct $\lambda_i$'s   among $i \geq 2$.
In section  8  we make a local characterization of GRS with harmonic Weyl curvature and exactly  two distinct $\lambda_i$'s    among $i \geq 2$, and finish the proof of Theorem \ref{locals}. In section 9 we classify complete steady and expanding GRS with harmonic Weyl curvature, proving Theorem 2 and Theorem 3.

\section{Preliminaries}

In this section we recall some known results about a GRS with harmonic Weyl curvature.
An $n$-dimensional Riemannian manifold $(M^n, g)$, $n\geq 4$,  is said to have harmonic Weyl curvature if the divergence $\delta W $  of the Weyl tensor $W$ is zero, which is equivalent to the Schouten tensor $S:=\frac{1}{n-2} (r  - \frac{R}{2(n-1)}g)$ being a Codazzi tensor, written in local coordinates as $\nabla_k S_{ij}  = \nabla_i S_{kj}$, \cite[Chap.16]{Be}.
A gradient Ricci soliton is real analytic in harmonic coordinates \cite{Iv}.
Below we shall denote the Ricci tensor as $r$ or $R( \cdot , \cdot )$.

 \begin{lemma} [Theorem 2.1 in \cite{FG}] \label{threesolx}
Let  $(M^n, g, f)$ be a GRS with harmonic Weyl curvature. The Riemannian curvature tensor satisfies
\begin{eqnarray*} \label{solba}
R(X, Y, Z,\nabla  f ) &=& \frac{1}{n - 1} R(X,\nabla  f )g(Y, Z) - \frac{1}{n - 1} R(Y,\nabla f )g(X, Z)\\
&=& \frac{1}{2(n - 1)} dR(X) g(Y, Z) - \frac{1}{2(n - 1)} dR(Y)g(X, Z).
\end{eqnarray*}
\end{lemma}

The next lemma comes from \cite{CC2}. See also Lemma 2.3 in \cite{Ki}.

 \begin{lemma} \label{threesolbx}
Let  $(M^n, g, f)$ be a GRS with harmonic Weyl curvature. For a point $p$ in $f^{-1}(c)$ with $\nabla f(p) \neq 0$,
let $\Sigma_c^0$ be the connected component of $f^{-1}(c)$. Then the following assertions hold:

{\rm (i)} Where $\nabla f \neq 0$,  $E_1 := \frac{\nabla f }{|\nabla f  | }$ is an eigenvector field of $r$.

{\rm (ii)} $ |\nabla f|$  is constant on $\Sigma_c^0$.

\smallskip
{\rm (iii)} There is a function $s$ locally defined in $\{ p \in M \ | \ \nabla f(p) \neq 0\}$ by   $s(x) = \int  \frac{   d f}{|d f|} $, so that  $ \ ds =\frac{   d f}{|d f|}$ and $E_1 = \nabla s$.

\smallskip
{\rm (iv)}  $R({E_1, E_1})$ is constant on $\Sigma_c^0$.

\smallskip
{\rm (v)}  Near a point in $\Sigma_c^0$, the metric $g$ can be written as
\begin{eqnarray} \label{solba76}
 g= ds^2 +  \sum_{i,j > 1} g_{ij}(s, x_2, \cdots  x_n) dx_i \otimes dx_j,
\end{eqnarray}
 where    $x_2, \cdots  x_n$ is a local
 coordinates system on $\Sigma_c^0$.

\smallskip
{\rm (vi)}  $\nabla_{E_1} E_1=0$.
\end{lemma}

\begin{proof}
Lemma 2.3 of \cite{Ki} proves the case when $c$ is a regular value of $f$. For our case,
let $\Sigma$ be the set of points which can be connected by a continuous path to $p$ in $\Sigma_c^0$.  $ |\nabla f  | $ is constant on $\Sigma$. It is easy to see that $\Sigma$ is open and closed in $\Sigma_c^0$ so that  $\Sigma=\Sigma_c^0$. This argument helps in proving (ii) and (iv).
\end{proof}

 For a point $x$ in $M$, let  $E_r(x)$ be the number of distinct eigenvalues of the Ricci tensor $r_x$,
and set $M_r = \{    x \in M \  |  \  E_r {\rm \ is \ constant \ in \ a \ neighborhood of \ } x \}$, following Derdzi\'{n}ski  \cite{De}, so that $M_r$ is an open dense subset of $M$.
 Then we have;

\begin{lemma}[Lemma 2.4 and Lemma 2.8 in \cite{Ki}] \label{abc60x} For a Riemannian metric $g$ of dimension $n \geq 4$ with harmonic Weyl curvature, consider orthonormal vector fields $E_i$, $i=1, \cdots n$ such that
$R(E_i, \cdot ) = \lambda_i g(E_i, \cdot)$. Then the followings hold
in each connected component of $M_r$;

\bigskip
\noindent {\rm (i)}
 $ \ (\lambda_j - \lambda_k ) \langle \nabla_{E_i} E_j, E_k \rangle   + {E_i} \{(r- \frac{R}{2n-2}g)(E_j, E_k)   \} \hspace{1cm} \\
   \hspace{1cm}  =(\lambda_i - \lambda_k ) \langle \nabla_{E_j} E_i, E_k\rangle + {E_j} \{(r- \frac{R}{2n-2}g)(E_k, E_i)   \},  $
 for any $i,j,k =1, \cdots n$.

\medskip
\noindent {\rm (ii)}  If $k \neq i$ and $k \neq j$,
$ \ \ (\lambda_j - \lambda_k ) \langle \nabla_{E_i} E_j, E_k\rangle=(\lambda_i - \lambda_k ) \langle \nabla_{E_j} E_i, E_k\rangle .$

\smallskip
\noindent {\rm (iii)} Given distinct eigenfunctions $\lambda$ and $ \mu$ of the Ricci tensor $r$ and local vector fields $v$ and $ u$ such that  $r v = \lambda v$, $ru = \mu u$ with $|u|=1$, it holds

$ \ \ \ \ \  v(\mu- \frac{R}{2(n-1)} ) = (\mu - \lambda) <\nabla_u u, v > $.

\smallskip
\noindent {\rm (iv)} For each eigenfunction $\lambda$ of $r$, the $\lambda$-eigenspace distribution is integrable and its leaves are totally umbilical submanifolds of $M$.

\end{lemma}

\begin{lemma}[Lemma 2.5 in \cite{Ki}] \label{aa8}
Let $(M^n,g,f)$  be an $n$-dimensional gradient Ricci soliton with harmonic Weyl curvature and non constant $f$.
For any point $p$ in the open dense subset $M_{r} \cap \{ \nabla f \neq 0  \}$ of $M^n$,
there is a neighborhood $U$ of $p$ where there exist  orthonormal Ricci-eigen vector fields $E_i$, $i=1, \cdots  , n$  such  that for all the eigenspace distributions $D_1, \cdots, D_k$ of ${r}$ in $U$,

  \medskip

 {\rm (i)}  $E_1= \frac{\nabla f}{|\nabla f| }$ is in $D_1$,

 {\rm (ii)} for $i>1$, $E_i$ is tangent to smooth level hypersurfaces of $f$,

 {\rm (iii)} let $d_l$ be the dimension of $D_l$ for $l=1, \cdots, k$, then $E_1,  \cdots , E_{d_1}  \in D_1$,   $ \ \ E_{d_1 +1}   ,  \cdots, E_{d_1+ d_2}  \in D_2 $, $\cdots$ , and    $E_{d_1 + \cdots +d_{k-1}+1}   ,  \cdots, E_{n}  \in D_k $.
\end{lemma}

These local orthonormal Ricci-eigen vector fields $E_i$ of Lemma \ref{aa8} shall be called an {\it adapted frame field} of $(M, g, f)$.
We set  $ \zeta_i:= - \langle   \nabla_{E_i}  E_i ,  E_1  \rangle=  \langle    E_i , \nabla_{E_i}  E_1  \rangle$, for $i >1$. By (\ref{grs}), $\nabla_{E_i}  E_1 = \nabla_{E_i} (\frac{\nabla f}{  | \nabla f |}) =   \frac{ -  R({E_i}, \cdot)+  \lambda g( {E_i}, \cdot  ) }{  | \nabla f |} $. So we may write:
\begin{equation} \label{lambda06ax}
\nabla_{E_i}  E_1 =   \zeta_i E_i     \ \    {\rm where}   \   \zeta_i =     \frac{- R(E_{i}, E_i)  + \lambda  }{| \nabla f|}.
\end{equation}
Due to Lemma \ref{threesolbx}, in a neighborhood of a point $p \in  \{ \nabla f \neq 0  \}$, $f$  may be considered as a function of the variable $s$ only, and $f^{'} := \frac{df}{ds} = |\nabla f|$.

\begin{lemma} \label{abc60byx} Let  $(M^n, g, f)$ be a GRS with harmonic Weyl curvature.
The Ricci eigenfunctions $\lambda_i$ associated to an adapted frame field $E_i$ in  $M_{r} \cap \{ \nabla f \neq 0  \}$
are constant on a connected component of a regular level hypersurface $\Sigma_c$ of $f$, and so depend on the local variable  $s$ only. Moreover, $\zeta_i$, $i=2, \cdots, n$, in  $\rm{(\ref{lambda06ax})}$ also depend on $s$ only.
 In particular, we have
$E_i (\lambda_j) = E_i (\zeta_k)= 0$ for $i,k >1$ and any $j$.
\end{lemma}
\begin{proof}
Lemma 2.7 of \cite{Ki} gives the proof in the four dimensional case. Similar argument can be given in higher dimension. Or, one can follow the proof of Lemma 3 in \cite{Sh}.
\end{proof}

%In this article we need to study about  real analytic objects.A function $H(x_1 , \cdots , x_{k-1} ; x_k) $ of $k$ variables is called a distinguished polynomial if it has the form$H(x_1 , \cdots , x_{k-1} ; x_k)  = x_k^m  + A_1 x_k^{m-1}+ \cdots +  A_{m-1} x_k^1  + A_m  $, where $A_1 , \cdots , A_m$ are functions of  $x_1 , \cdots , x_{k-1}$ only.  We have the Weierstrass theorem as follows.{\bf  The  Weierstrass Preparation theorem} \cite[Theorem 5.2.1]{KP}: Let  $f(x_1 , \cdots ,x_k) $ be a  real analytic function in a neighborhood of the origin in $\mathbb{R}^k$ with  $f(0, \cdots, 0, x_k )\neq 0.$ Then $f$ may be written in the form $f= H \cdot U$, where $H$ is a distinguished polynomial and the analytic function $U$ does not vanish in a neighborhood of the origin.

 \section{Gradient Ricci soliton with harmonic Weyl curvature in an  adapted frame field}
For an adapted frame field $E_i$,  from (\ref{grs}), (\ref{lambda06ax}) and Lemma \ref{threesolbx} (vi) we have
\begin{eqnarray}
f^{''}  =  -\lambda_1  + \lambda, \hspace{1.2cm} \label{x1} \\
\zeta_i f^{'} =  -\lambda_i  + \lambda, \ \ \    i >1. \label{x1z}
\end{eqnarray}
%Recall $\nabla_{E_1}  E_1 =0$ from Lemma \ref{threesolbx} {\rm (vi)} and $\nabla_{E_i} E_1 = \zeta_i E_i$ from (\ref{lambda06ax}).
  By direct curvature computation we get $R_{1ii1} =-\zeta_i^{'}  -  \zeta_i^2.$
Lemma \ref{threesolx} gives
\begin{eqnarray} \label{x2}
R_{1ii1} =-\zeta_i^{'}  -  \zeta_i^2  = \frac{\lambda_{1}}{n-1} = \frac{R^{'}}{2(n-1) f^{'}},  \ \ \    i >1.
\end{eqnarray}

\begin{lemma} \label{gg}
Suppose that there is $i_0>1$ with $\zeta_{i_0} =0$
in an  adapted frame field $E_i$ on an open subset $U$ of a gradient Ricci soliton $(M, g, f)$ with harmonic Weyl curvature.

 Then  $(U, g)$ is locally isometric to a Ricci flat metric  with $\lambda=0$ or
a Riemannian product $ (N_1^{k}, g_1 ) \times (N_2^{n-k}, g_2)$, where $g_1$ is Ricci flat and $r_{g_2} =  \lambda g_2 $, $\lambda \neq 0$, $1 \leq k \leq n-2$. Moreover $f$  is the pull-back of a function  on $N_1^k$, denoted by $f$ again, and    $ (N_1^{k}, g_1, f )$  is itself a gradient Ricci soliton.
We have $f = \frac{\lambda s^2}{2}$ modulo a constant, where $s$ is a function such that $\nabla s =  \frac{\nabla f }{| \nabla f |}$.
\end{lemma}
 \begin{proof}We first suppose that there is $i_0>1$ with $\zeta_{i_0} =0$, and  $i_1>1$ with $\zeta_{i_1} \neq 0$.
Then from (\ref{x2}), $\zeta_i^{'}  +  \zeta_i^2=0$ for any $i >1$ so that
$\zeta_i(s) = \frac{1}{ s + c_i} $ or $\zeta_i(s) = 0$.
We may assume that $E_i$ are ordered so that $\zeta_i = \frac{1}{ s + c_i}  $ for $ 2 \leq i \leq k$,  and $\zeta_i =0  $ for $ k+1 \leq i \leq n$.
From (\ref{x2})  $ \lambda_{1} = R^{'}=0$, $f^{''} = \lambda$, $f^{'} = \lambda s + c_1$ for a constant $c_1$.
 For $ 2 \leq i \leq k$, (\ref{x1z}) gives $\lambda_i =  \lambda-\zeta_i f^{'}=\lambda-\frac{1}{ s + c_i} ( \lambda s + c_1 ) =\frac{\lambda {c}_i -c_1}{ s + c_i} $.
 For $ k+1 \leq i \leq n$,  $\lambda_i =  \lambda $.
We have $R =   \sum_{i=1}^n \lambda_i =  (n-k)  \lambda +   \sum_{i=2}^k \frac{\lambda {c}_i -c_1}{ s + c_i}  $. As $R$ is a constant, considering distinct $c_i$'s if any, we get $\lambda {c}_i -c_1 =0 $  for $ 2 \leq i \leq k$ so that $\lambda_i=0$.
 Now   $R = (n-k) \lambda  $.
If  $\lambda=0$, then $g$ is Ricci flat.
 Suppose that $\lambda \neq 0$. By Lemma \ref{abc60x} (iii) and (iv), both $\lambda_i$-eigenspaces, $i=1, n$, are integrable and totally geodesic. By the de Rham decomposition theorem, $(M, g)$ is locally isometric to a Riemannian product $ (N_1^{k}, g_1 ) \times (N_2^{n-k}, g_2)$, where $g_1$ is Ricci flat and $r_{g_2} =  \lambda g_2 $. If $k=n-1$, then  $\lambda=0$, a contradiction. So, $2 \leq k \leq n-2$. We may write $f^{'} = \lambda s $ by a translation of $s$ by a constant, and $f = \frac{\lambda s^2}{2}$ modulo a constant.
As $g$ is locally the Riemannian product of $g_1$ and $g_2$, from $\nabla^{g} d f = - r_{g} +\lambda g$,  we  get
$\nabla^{g_1} d f = - r_{g_1} +\lambda g_1= \lambda g_1$, i.e. $ (N_1^{k}, g_1, f )$  is itself a gradient Ricci soliton.

\medskip
If $\zeta_i=0$ for all $i >1$, then $\lambda_i=\lambda$, $\lambda_1= 0$ and $f^{''} = \lambda$. If  $\lambda=0$, then $g$ is Ricci flat.
 Suppose that $\lambda \neq 0$. By Following the argument in the above paragraph,
$(M, g)$ is locally isometric to a Riemannian product $ (N_1^{1}, g_1 ) \times (N_2^{n-1}, g_2)$, where $g_1$ is one-dimensional metric  and $r_{g_2} =  \lambda g_2 $. Summarizing, we proved the lemma.
\end{proof}

\medskip
 Now suppose that $\zeta_i $ is not the zero function for any $i >1$. By Lemma \ref{abc60byx}, $\zeta_i = \zeta_i(s)$.
 Setting $\zeta_i = \frac{u_i^{'}}{u_i}$ for a nowhere-zero function $u_i (s)$,  we have $\frac{u_i^{''}}{u_i}=\frac{u_2^{''}} {u_2}$ for $ i \geq 2$ from (\ref{x2}).
Viewing this as an ordinary differential equation for $u_i$, we solve it by reduction of order by  setting $u_i = u_2 k_i$ for a nowhere-zero function $k_i :=k_i(s)$. Now $u_i'  = u_2^{'} k_i +  u_2 k_i^{'}$. $u_i''  = u_2^{''} k_i + 2 u_2^{'} k_i^{'} +  u_2 k_i''$.
Then
$ \frac{u_2^{''}}{u_2}  =\frac{ u_2^{''} k_i + 2 u_2^{'} k_i^{'} +  u_2 k_i''}{u_2 k_i} $. So, $2 u_2^{'} k_i^{'} +  u_2 k_i'' =0$. Integration gives
 $ k_i^{'} = \frac{d_i}{ (u_2)^2 }$ for a constant $d_i$, and  $k_i = e_i +\int_{s_0}^s \frac{d_i}{ (u_2)^2 } ds \ $ for constants $e_i$ and $s_0$. Then $d_i =0$ iff $\zeta_i =  \frac{u_i^{'}}{u_i} =   \frac{u_2' }{u_2  } =\zeta_2.  $

 Suppose $d_i \neq 0$. Then  $\frac{  k_i'}{k_i }  = \frac{1}{(u_2)^2 (c_i +\int_{s_0}^s \frac{ds}{ (u_2)^2 }) } $ for $c_i =  \frac{{ e_i} }{d_i}$.
As $ \zeta_i =  \frac{u_2^{'} k_i +  u_2 k_i'}{u_2 k_i } =   \frac{u_2^{'} }{u_2  } + \frac{  k_i'}{k_i } $, we get
$ \zeta_i =   \frac{u_2^{'} }{u_2  } +  \frac{1}{(u_2)^2 (c_i +\int_{s_0}^s \frac{ds}{ (u_2)^2 }) }.$
We define an equivalence relation $\sim$  on the set  $\{ 2, 3, \cdots, n \} $ as below
 \begin{eqnarray} \label{x3j0dk}
i \sim j \ \  {\rm  iff} \ \  \zeta_i = \zeta_j.
\end{eqnarray}
and let $[i]$ denote the equivalence class of $i$. Clearly $i \sim j$  iff $ c_i = c_j$.
Set
 \begin{eqnarray} \label{x3j0da2}
 h(s)=\int_{s_0}^s \frac{ds}{ (u_2)^2 },
\end{eqnarray}
and we get
 \begin{eqnarray} \label{x3j0d}
 \begin{cases}
 \ \  \zeta_i =   \frac{u_2^{'} }{u_2  }  \ \  \ {\rm when}  \ \ [i]= [2],\\
 \  \zeta_i =   \frac{u_2^{'} }{u_2  } +  \frac{h^{'}}{ c_i + h }  \ \  \ {\rm when}  \ \ [i] \neq [2].
 \end{cases}.
\end{eqnarray}
Note that $c_i + h(s)$ is nowhere zero as $k_i$ is.

\begin{lemma}\label{77bx01}Let  $(M^n, g, f)$ be a GRS with harmonic Weyl curvature.
In an adapted frame field $\{ E_j\}$
in an open subset $W$ of $M_{r} \cap \{ \nabla f \neq 0  \}$, we have for $i \geq 2$;
\begin{eqnarray}
 \ \ \ \lambda_{i}
= -\zeta_i^{'}  -  \zeta_i^2  + \sum_{j=2, j \neq i}^n \{  -\zeta_i \zeta_j + E_i \Gamma_{jj}^i  - E_j \Gamma_{ij}^i -(\Gamma_{ii}^j)^2  -(\Gamma_{jj}^i)^2 \} \label{y01}  \\
+\sum_{j=2, j \neq i}^n \sum_{k \neq 1, i, j } \{\Gamma_{jj}^k \Gamma_{ik}^i
   -\Gamma_{ij}^k\Gamma_{jk}^i  -( \Gamma_{ij}^k-\Gamma_{ji}^k) \Gamma_{kj}^i \}. \hspace{1.3cm}  \nonumber
\end{eqnarray}
If $\zeta_i =   \zeta_j \neq \zeta_k $ for $i, j, k \in \{2,3, \cdots , n \}$, then setting  $ \Gamma_{ij}^k = \langle \nabla_{E_i} E_j, E_k \rangle $,
\begin{eqnarray} \label{y03}
\Gamma_{ij}^k=0  \ \ \ \  {\rm  and}  \ \ \ \   \Gamma_{1i}^k =0.
\end{eqnarray}

\end{lemma}

%$\nabla_{E_1}  E_1 =0$ from Lemma \ref{threesolbx} {\rm (vi)} and  $\nabla_{E_i} E_1 = \zeta_i E_i$ from (\ref{lambda06ax}).
% We get

%$R_{1221}= < \nabla_{E_1} \nabla_{E_2} E_2 - \nabla_{E_2} \nabla_{E_1} E_2  - \nabla_{[E_1, E_2]} E_2,  E_1 >= < \nabla_{E_1} (-\zeta_2 E_1 + \sum_{s>1}  \Gamma_{22}^s E_s  ) - \nabla_{E_2} (   \sum_{s\geq 1}  \Gamma_{12}^s E_s )   - \nabla_{ \sum_{s\geq 1}  \Gamma_{12}^s E_s - \zeta_2 E_2} E_2,  E_1 >\\= < -\nabla_{E_1} (\zeta_2 E_1  )  +\nabla_{ \sum_{s\geq 1}  \zeta_2 E_2} E_2,  E_1 >=-\zeta_2^{'}  -\zeta_2^2$.  So,   $ \ \ \  R_{1ii1} =-\zeta_i^{'}  -\zeta_i^2 $.

% $R_{1213}= < \nabla_{E_1} \nabla_{E_2} E_1  -  \nabla_{E_2} \nabla_{E_1} E_1  - \nabla_{[E_1, E_2]} E_1,  E_3 > \\ = < \nabla_{E_1} (\zeta_2 E_2)   - \nabla_{[E_1, E_2]} E_1,  E_3 >= (\zeta_2 - \zeta_3 )\Gamma_{12}^3$.

%$R_{2332}= < \nabla_{E_2} \nabla_{E_3} E_3  -  \nabla_{E_3} \nabla_{E_2} E_3  - \nabla_{[E_2, E_3]} E_3,  E_2 > \\= < \nabla_{E_2}( \sum_{s\geq 1}\Gamma_{33}^sE_s )   -  \nabla_{E_3} ( \sum_{s\geq 1}\Gamma_{23}^sE_s ) - \nabla_{ \sum_{s\geq 1}(\Gamma_{23}^s-\Gamma_{32}^s) E_s} E_3,  E_2 >\\= -\zeta_2 \zeta_3 +E_2 \Gamma_{33}^2  - E_3 \Gamma_{23}^2+ \sum_{s\geq 2}\Gamma_{33}^s\Gamma_{2s}^2 -  \sum_{s\geq 2}\Gamma_{23}^s\Gamma_{3s}^2-\sum_{s\geq 1}(\Gamma_{23}^s-\Gamma_{32}^s) <   \nabla_{ E_s} E_3,  E_2 >\\= -\zeta_2 \zeta_3 + E_2 \Gamma_{33}^2  - E_3 \Gamma_{23}^2 - (\Gamma_{22}^3 )^2 - (\Gamma_{33}^2 )^2 + \sum_{i =4,5 }\Gamma_{33}^i \Gamma_{2i}^2   -\Gamma_{23}^i \Gamma_{3i}^2  -( \Gamma_{23}^i-\Gamma_{32}^i) \Gamma_{i3}^2$.

\begin{proof}
For distinct $i, j \in \{2,3, \cdots , n \}$, we can directly compute

$R_{ijji}:= R(E_i, E_j, E_j, E_i)
= -\zeta_i \zeta_j + E_i \Gamma_{jj}^i  - E_j \Gamma_{ij}^i -(\Gamma_{ii}^j)^2  -(\Gamma_{jj}^i)^2+ \sum_{k \neq 1, i, j }\Gamma_{jj}^k \Gamma_{ik}^i
   -\Gamma_{ij}^k\Gamma_{jk}^i  -( \Gamma_{ij}^k-\Gamma_{ji}^k) \Gamma_{kj}^i
$. From (\ref{x2}) we get (\ref{y01}).

If $\zeta_i =   \zeta_j \neq \zeta_k $ for $i, j, k \in \{2,3, \cdots , n \}$, then $\lambda_i =   \lambda_j \neq \lambda_k $.
 From Lemma  \ref{abc60x} (iii) and Lemma \ref{abc60byx},
 $ \langle  \nabla_{E_i} E_i ,   E_k\rangle=0$ and  $ \langle  \nabla_{E_i+ E_j} (E_i + E_j),   E_k\rangle=0$. By   Lemma  \ref{abc60x} (iv), $ \langle  \nabla_{E_i} E_j   -\nabla_{E_j} E_i  ,   E_k\rangle=0$. So, $\Gamma_{ij}^k=0$.
From Lemma  \ref{abc60x} {\rm (ii)}, $(\lambda_i - \lambda_k) \langle \nabla_{E_1} E_i, E_k\rangle=(\lambda_1 - \lambda_k ) \langle \nabla_{E_i} E_1, E_k\rangle $.  As $\langle \nabla_{E_i} E_1, E_k\rangle=0 $ from (\ref{lambda06ax}),  $\langle\nabla_{E_1} E_i, E_k\rangle=0$.
This proves (\ref{y03}).

%Now, $ \langle  \nabla_{E_i} E_j, E_i \rangle=-\langle  \nabla_{E_i} E_i, E_j \rangle=0$,$\langle  \nabla_{E_i} E_j, E_j \rangle=0  $. And $\langle\nabla_{E_i} E_j, E_1 \rangle = -\langle   \nabla_{E_i} E_1 ,  E_j\rangle =0 $. So,  $\nabla_{E_i} E_j= \sum_{k \neq 1,i,j}\Gamma_{ij}^k E_k$.  Clearly  $\Gamma_{ij}^k =- \Gamma_{ik}^j $.

%$R_{ii}= R_{i11i}+ \sum_{j=2}^n  R_{ijji} \\= -\zeta_i^{'}  -  \zeta_i^2  + \sum_{j=2, j \neq i}^n \{  -\zeta_i \zeta_j + E_i \Gamma_{jj}^i  - E_j \Gamma_{ij}^i -(\Gamma_{ii}^j)^2  -(\Gamma_{jj}^i)^2+ \sum_{k \neq 1, i, j }\Gamma_{jj}^k \Gamma_{ik}^i     -\Gamma_{ij}^k\Gamma_{jk}^i  -( \Gamma_{ij}^k-\Gamma_{ji}^k) \Gamma_{kj}^i \} $.

%?? $R_{2342}= < \nabla_{E_2} \nabla_{E_3} E_4  -  \nabla_{E_3} \nabla_{E_2} E_4  - \nabla_{[E_2, E_3]} E_4,  E_2 > \\= {E_2} \Gamma_{34}^2 -  {E_3}  \Gamma_{24}^2  +   \sum_{s=2,3,4,5} \Gamma_{34}^s \Gamma_{2s}^2-  \Gamma_{24}^s  \Gamma_{3s}^2  - ( \Gamma_{23}^s -  \Gamma_{32}^s)\Gamma_{s4}^2.$

\end{proof}

%$R_{1213}= < \nabla_{E_1} \nabla_{E_2} E_1  -  \nabla_{E_2} \nabla_{E_1} E_1  - \nabla_{[E_1, E_2]} E_1,  E_3 > \\ = < \nabla_{E_1} (\zeta_2 E_2)   - \nabla_{[E_1, E_2]} E_1,  E_3 >= (\zeta_2 - \zeta_3 )\Gamma_{12}^3$.

%\bigskip$R_{1243}= < \nabla_{E_1} \nabla_{E_2} E_4  -  \nabla_{E_2} \nabla_{E_1} E_4  - \nabla_{[E_1, E_2]} E_4,  E_3 > \\=< \nabla_{E_1} \sum_k \Gamma_{24}^k E_k  -  \nabla_{E_2}  \sum_k \Gamma_{14}^k E_k   - \nabla_{\Gamma_{12}^k  E_k  - \zeta_2 E_2} E_4,  E_3 >\\=\nabla_{E_1} \Gamma_{24}^3 +   \sum_k \Gamma_{24}^k   \Gamma_{1k}^3 + <  -  \nabla_{E_2}  \sum_k \Gamma_{14}^k E_k   - \nabla_{\Gamma_{12}^k  E_k  - \zeta_2 E_2} E_4,  E_3 >$.

 \section{Refined adapted Ricci-eigen vector fields}
 Let  $p_0  $  be a point in the open dense subset $M_{r} \cap \{ \nabla f \neq 0  \}$ of  a gradient Ricci soliton $M^n$ with harmonic Weyl curvature. Say $f(p_0) = c$.
Suppose that there is  an adapted frame field $E_1,  \cdots,  E_n$ and the functions $\lambda_i$ with $R(E_i, \cdot ) = \lambda_i g(E_i, \cdot)$ on an open neighborhood $U$ of $p_0$.
We consider the local one-parameter group action $ \theta(p, t )$ associated to $E_1= \frac{\nabla f}{|\nabla f| }$ near $p_0$, i.e.  $\frac{d ( \theta(p, t))}{dt}|_{t= 0} =  E_1 (p)$.
We suppose that
$\theta(p, t)$ is defined for  $ (p, t) \in ( V \cap  f^{-1}(c)) \times (-\varepsilon, \varepsilon ) $ for an open neighborhood   $V$ of $p_0$ in $U$ and some number $\varepsilon>0$.
Set $W:=V \cap  f^{-1}(c)$. For notational convenience we often denote $ \theta(p, t )$ by  $ \theta_p( t )$  or $ \theta_t(p) $. For the map $\theta_t : p \mapsto \theta_p( t )$, we denote the derivative as ${\theta_{t}}_*$.
We use the Lie derivative  $L_{E_1 } E_i (q) := (L_{E_1 } E_i) (q) =   \lim_{h \rightarrow 0} \frac{1}{h} \{ {E_i}(q)- {\theta_{h}}_* {E_i}({\theta_{-h} (q)}) \} $.

\begin{lemma} \label{L61}
For $i \geq 2$,
${\theta_{t}}_* ({E_i}(p))$  belongs to the $\lambda_i$-eigenspace in $T_{\theta_{t}(p)}M$ for any $p \in W$ and $t \in (-\varepsilon, \varepsilon )$.
\end{lemma}

\begin{proof}
Let $q $ be a point in  $\theta( W \times (-\varepsilon, \varepsilon )) $.  For  $i \geq 2$, $j \geq 1$  define $\beta_{ij}(t)  =  g   ( {\theta_{t}}_* {E_i}({\theta_{-t} (q)})  , {E_j}(q)  )$ for   $t $ such that $\theta_{-t} (q)$ also lies $\theta( W \times (-\varepsilon, \varepsilon )) $. We compute
\begin{eqnarray*}\frac{d}{dt} \{ {\theta_{t}}_* {E_i}({\theta_{-t} (q)}) \} =  \lim_{h \rightarrow 0} \frac{1}{h} \{{\theta_{t}}_*  {\theta_{h}}_* {E_i}({\theta_{-t-h} (q)}) -   {\theta_{t}}_* {E_i}({\theta_{-t} (q)})\}   \hspace{1.5cm}  \\
=  {\theta_{t}}_* [ \lim_{h \rightarrow 0} \frac{1}{h} \{  {\theta_{h}}_* {E_i}({\theta_{-t-h} (q)}) -   {E_i}({\theta_{-t} (q)})\}]
=  - {\theta_{t}}_* [ L_{E_1 } E_i (\theta_{-t} (q))]. \end{eqnarray*}

%$\beta_{ij}^{'} (t)  =g   (  \lim_{h \rightarrow 0} \frac{1}{h} \{{\theta_{-t}}_*  {\theta_{-h}}_* {E_i}({\theta_{t+h} (q)}) -   {\theta_{-t}}_* {E_i}({\theta_{t} (q)})\}  , {E_j}(q)  )\\ =g   (  {\theta_{-t}}_* [ \lim_{h \rightarrow 0} \frac{1}{h} \{  {\theta_{-h}}_* {E_i}({\theta_{t+h} (q)}) -   {E_i}({\theta_{t} (q)})\}]  , {E_j}(q)  )\\=  g   (  {\theta_{-t}}_* [ L_{E_1 } E_i (\theta_{t} (q))]  , {E_j}(q)  ) $.

By Lemma \ref{77bx01},  $ \Gamma_{1i}^l=0$ if $l \notin [i]$. So,  we have $L_{E_1 } E_i = \nabla_{E_1 } E_i  - \nabla_{E_i } E_1 = \sum_{l \in [i]. l \neq i}   \Gamma_{1i}^l E_l -  \zeta_i E_i   $. Then,
\begin{eqnarray*}
\beta_{ij}^{'} (t)  & =  -g   (  {\theta_{t}}_* [ L_{E_1 } E_i (\theta_{-t} (q))]  , {E_j}(q)  )\hspace{6.4cm} \\
&= - g   (  {\theta_{t}}_* [  \sum_{l \in [i]. l \neq i}   \Gamma_{1i}^l( \theta_{-t} (q)) {E_l }({\theta_{-t} (q)}) -  \zeta_i( \theta_{-t} (q)) {E_i}({\theta_{-t} (q)})]  , {E_j}(q)  ) \  \ \   \\
&=  - g  (  [  \sum_{l \in [i]. l \neq i}   \Gamma_{1i}^l (\theta_{-t} (q))  {\theta_{t}}_*  {E_l }({\theta_{-t} (q)})-  \zeta_i (\theta_{-t} (q))  {\theta_{t}}_* {E_i}({\theta_{-t} (q)}) ]  , {E_j}(q)  )  \\
&= -\sum_{l \in [i]. l \neq i}  \Gamma_{1i}^l(\theta_{-t} (q)) \beta_{lj} (t)+ \zeta_i (\theta_{-t} (q))  \beta_{ij} (t). \hspace{3.9cm}
\end{eqnarray*}

We fix $ j \notin [i]$.   We view
$\beta_{ij}^{'}
= - \sum_{l \in [i]. l \neq i}   \Gamma_{1i}^l(\theta_{-t} (q)) \beta_{lj}   +\zeta_i(\theta_{-t} (q)) \beta_{ij}
 $
as  a system of  first order ordinary differential equations  for $r$ functions $\beta_{i_1 j} (t)$, $\beta_{i_2j} (t)$,  $ \cdots,   \beta_{i_r j} (t)$ where  $E_{i_1}, \cdots ,  E_{i_r}$ is the basis of the $\lambda_i$-eigenspace.

As $\beta_{i_1j}(0)= \cdots = \beta_{i_r j}(0)  = 0 $,  by the uniqueness of the solution of the ODE,   $\beta_{ij} (t) =  g  ( {\theta_{t}}_* {E_i}({\theta_{-t} (q)}) , {E_j}(q)  )=0$.
If $q = \theta_{t_0} (p) $ for some $t_0 \in (-\varepsilon, \varepsilon )$ and  $p \in W$, then ${\theta_{t}}_* {E_i}({\theta_{-t} (\theta_{t_0} (p))}) $ is orthogonal to  ${E_j}(\theta_{t_0} (p)) $. When $t=t_0$,
 ${\theta_{t_0}}_* {E_i}( {\theta_{-t_0} (\theta_{t_0}(p)))={\theta_{t_0}}_*( {E_i}(p}))  $ is orthogonal to  ${E_j}(\theta_{t_0}(p)) $.
 These $E_j(\theta_{t_0}(p))$,  $j \notin [i]$, span the subspace orthogonal to the $\lambda_i$-eigenspace in  $T_{\theta_{t_0}(p)}M$.
So, ${\theta_{t_0}}_* ({E_i}(p))$  belongs to the $\lambda_i$ -eigenspace in  $T_{\theta_{t_0}(p)}M$ for $p \in W$ and $t_0 \in (-\varepsilon, \varepsilon )$.
\end{proof}

The function $\zeta_i$ is defined via  (\ref{lambda06ax}) and the function $s$ is defined modulo a constant in Lemma \ref{threesolbx}. We may set $s(p)  = 0 $ for $p \in W$.
 By Lemma \ref{abc60byx},  $\zeta_i$ depends only on $s$.    As  $\frac{d}{dt}  \theta_t (p) = E_1 ( \theta_t (p) )= \nabla s ( \theta_t (p) ) $, we have   $s( \theta_t  (p)) = t$.
For convenience we write
$ \zeta_i(\theta_{t} (p))=\zeta_i(s(\theta_{t} (p))) =\zeta_i( t) $.

\begin{lemma} \label{L42} Following the terminology in the beginning paragraph of this section,
we let $i \geq 2$ and $\zeta_i(s)$  be as in the above paragraph.
For $p  \in  W$, $t \in (-\varepsilon, \varepsilon )$, let $F_i$ be the vector fields defined on $ \theta(W \times  (-\varepsilon, \varepsilon ) )$  by $F_i (\theta_{t} (p)) =e^{- \int_{0}^t   \zeta_i(v) dv }  {\theta_{t}}_* ({E_i}(p))$.
If  $E_{i_1}(p), \cdots ,  E_{i_r}(p)$ is the basis of the $\lambda_i$-eigenspace in $T_{p}M$, then
$F_{i_1}(\theta_{t} (p)), \cdots ,  F_{i_r}(\theta_{t} (p))$
 form an orthonormal basis  for  the $\lambda_i$-eigenspace in $T_{\theta_{t}(p)}M$.

\end{lemma}
\begin{proof} Given $q \in \theta( W \times (-\varepsilon, \varepsilon ))$,
for $i, j \in \{i_1, \cdots, i_r \}$,
we define  $\gamma_{ij}(t) =  g  ( {\theta_{t}}_* {E_i}(\theta_{-t} (q))  , {\theta_{t}}_* {E_j}(\theta_{-t} (q)) )$,  for   $t $ such that $\theta_{-t} (q)$ also lies $\theta( W \times (-\varepsilon, \varepsilon )) $.
Then $\gamma_{ij}(0)  =\delta_{ij}$. As in the proof of Lemma \ref{L61},
%$\gamma_{ij}^{'} (0) =  g_p   ( L_{E_1} E_i  , E_j(q)  ) + g_p   ( E_i  , L_{E_1} E_j(q)  ) =- 2 \zeta_i(q) \delta_{ij}$.
\begin{eqnarray*}
\gamma_{ij}^{'}(t)  =  -g   ( {\theta_{t}}_* [ L_{E_1 } E_i (\theta_{-t} (q))] , {\theta_{t}}_* {E_j}(\theta_{-t} (q))  )-   g   ( {\theta_{t}}_* {E_i}(\theta_{-t} (q))  , {\theta_{t}}_* [ L_{E_1 } E_j (\theta_{-t} (q))]  )\\
=-  g   ([  \sum_{l \in [i]. l \neq i}   \Gamma_{1i}^l(\theta_{-t} (q)) {\theta_{t}}_*  {E_l }(\theta_{-t} (q))-  \zeta_i(\theta_{-t} (q)) {\theta_{t}}_* {E_i}(\theta_{-t} (q)) ], {\theta_{t}}_* {E_j}(\theta_{-t} (q))  )   \\
-  g   ( {\theta_{t}}_* {E_i}(\theta_{-t} (q))  , [  \sum_{l \in [j]. l \neq j}   \Gamma_{1j}^l(\theta_{-t} (q)) {\theta_{t}}_*  {E_l }(\theta_{-t} (q))-  \zeta_j(\theta_{-t} (q)) {\theta_{t}}_* {E_j}(\theta_{-t} (q)) ]  ).
\end{eqnarray*}
As $\zeta_i = \zeta_j$ and $[i]=[j]$, we can get
\begin{eqnarray}
\label{r5}   \\
\gamma_{ij}^{'}(t)
=  2 \zeta_i(\theta_{-t} (q))\gamma_{ij}(t)- \sum_{l \in [i], l \neq i}   \Gamma_{1i}^l(\theta_{-t} (q)) \gamma_{lj}(t)
-\sum_{l \in [j], l \neq j}   \Gamma_{1j}^l(\theta_{-t} (q)) \gamma_{il}(t). \nonumber
\end{eqnarray}
 We view the above as a system of  first order ODE for  the $r^2 $ functions $\gamma_{ij} (t)$, $i, j \in \{ i_1, \cdots i_r \}$ in the variable $t$.
 One can check that $\tilde{\gamma}_{ij}(t) = \delta_{ij}e^{ \int_{0}^t   2 \zeta_i(\theta_{-u} (q)) du }  $ is a solution  of (\ref{r5}).
As  $\gamma_{ij}(0)  = \tilde{\gamma}_{ij}(0)=\delta_{ij}$,  by  the uniqueness of the solution of the ODE, $\gamma_{ij}(t)=  g   ( {\theta_{t}}_* {E_i}(\theta_{-t} (q))  , {\theta_{t}}_* {E_j}(\theta_{-t} (q))  ) =   \delta_{ij} e^{ \int_{0}^t   2 \zeta_i(\theta_{-u} (q)) du }  $ for any  $t \in (-\varepsilon, \varepsilon )$.
So,  $ e^{ \int_{0}^t  - \zeta_i(\theta_{-u} (q)) du } {\theta_{t}}_* {E_i}(\theta_{-t} (q))$, $i =i_1, \cdots, i_r, $ is an orthonormal basis   for  the $\lambda_i$-eigenspace in $T_qM$. This holds for any point $q$ in  $ \theta(W \times (-\varepsilon, \varepsilon ))$, so we may set   $q= \theta_{t}(p)$ for $p  \in W$.
Then  $e^{ \int_{0}^t -  \zeta_i(\theta_{-u} (\theta_{t} (p))) du }  {\theta_{t}}_* {E_i}(p)$, $i =i_1, \cdots , i_r $ is an orthonormal basis   for  the $\lambda_i$ -eigenspace in $T_{ \theta_{t} (p)} M$.
Now $e^{ \int_{0}^t  - \zeta_i(\theta_{-u} (\theta_{t}(p))) du }   = e^{ \int_{0}^t - \zeta_i( \theta_v(p) ) dv }$ where  $v= t-u$.

So,
 $F_i (\theta_{t} (p)) :=e^{- \int_{0}^t   \zeta_i(\theta_v(p)) dv }  {\theta_{t}}_*( {E_i}(p))$ is  an orthonormal basis  for  the $\lambda_i$-eigenspace in $T_{\theta_{t}(p)}M$.
\end{proof}
By Lemma \ref{L42}, we have a new adapted frame field $F_1:=E_1, F_2, \cdots,   F_n$
on the open subset $\theta(W \times (-\varepsilon, \varepsilon ))$ of $M$.
We compute, for $i >1$,
\begin{eqnarray*}
L_{F_1 }   F_i (\theta_t(p)) =   \lim_{h \rightarrow 0} \frac{1}{h} \{  {\theta_{-h}}_* {F_i}( \theta_{h+t} (p)) -   {F_i}(\theta_t(p))  \} \hspace{3cm} \\
=   \lim_{h \rightarrow 0} \frac{1}{h} \{  {\theta_{-h}}_*  (e^{ \int_{0}^{t+h}  - \zeta_i(\theta_v(p)) dv }  {\theta_{t+h}}_* {E_i}_{p})
-   e^{ \int_{0}^t  - \zeta_i(\theta_v(p)) dv }  {\theta_t}_* {E_i}_{p}\} \\
=   \lim_{h \rightarrow 0} \frac{1}{h} \{  e^{ \int_{0}^{t+h}-   \zeta_i(\theta_v(p)) dv }  {\theta_{t}}_* {E_i}_{p}
-   e^{ \int_{0}^t -  \zeta_i(\theta_v(p)) dv }  {\theta_t}_* {E_i}_{p}\}   \hspace{1.4cm} \\
= \lim_{h \rightarrow 0} \frac{1}{h} \{  e^{ \int_{0}^{t+h}   -\zeta_i(\theta_v(p)) dv }
-   e^{ \int_{0}^t  - \zeta_i(\theta_v(p)) dv }  \}{\theta_{t}}_* {E_i}_{p}  \hspace{2.4cm} \\
=- \zeta_i(\theta_{t} (p)) F_i (\theta_{t} (p)). \hspace{6.7cm}
\end{eqnarray*}

Simply writing, we have
$L_{F_1 }   F_i=- \zeta_i F_i .$
As adapted Ricci-eigen vector fields, $F_i$ satisfy (\ref{lambda06ax}), so
$L_{F_1 }   F_i = \nabla_{F_1 }   F_i  - \nabla_{F_i }   F_1 =  \nabla_{F_1 }   F_i  -\zeta_i F_i $.
By comparison we get
\begin{equation} \label{e07u}
\nabla_{F_1 }   F_i=0 \ \ {\rm for}  \ \  i >1.
\end{equation}
Summarizing above, we have
\begin{prop} \label{pop1}
In a  gradient Ricci soliton $(M, g, f)$ with harmonic Weyl curvature,
suppose that there is an adapted  frame field $E_1,  \cdots,  E_n$  on an open subset $V$  and that the local one-parameter group action $\theta_t (p)  = \theta(p, t )$ associated to $E_1= \frac{\nabla f}{|\nabla f|}$  is defined for  $ ( p, t) \in ( V \cap  f^{-1}(c)) \times(-\varepsilon, \varepsilon ) $ for some real number $c$.

Then there exists a new adapted frame field $F_1=E_1, F_2,  \cdots,  F_n$ on $( V \cap  f^{-1}(c)) \times(-\varepsilon, \varepsilon ) $ satisfying {\rm (\ref{e07u})} defined by
$F_i (\theta_{t} (p)) :=e^{- \int_{0}^t   \zeta_i( \theta_v(p) ) dv }  {\theta_{t}}_* ({E_i}(p))$ for $i>1$.
\end{prop}
\medskip
As $F_i (p) = E_i(p)$, it holds that on $( V \cap  f^{-1}(c)) \times(-\varepsilon, \varepsilon ) $,
\begin{equation} \label{e07u77}
F_i (\theta_{t} (p)) =e^{- \int_{0}^t   \zeta_i( \theta_v(p) ) dv }  {\theta_{t}}_* ({F_i}(p)).
\end{equation}
Note that without any previously given $E_i$,
we may define $F_i$  by (\ref{e07u77}), given $F_i(p)$ for $p \in  V \cap  f^{-1}(c)$ because $\zeta_i$ is  independent  on the choice of $E_i$ from (\ref{lambda06ax}).

The frame field $\{F_i\}$  in Proposition \ref{pop1}  shall be called a {\it refined} adapted frame field.
For any point $p_0$ in $M_{r} \cap \{ \nabla f \neq 0  \}$, there exists a neighborhood of $p_0$ equipped with a  refined adapted frame field.

\section{Computing $\Gamma_{ij}^k $ in a refined adapted frame field $\{ F_i\}$}

In this section we compute  $\Gamma_{ij}^k= g( \nabla_{F_i} F_j, F_k) $. We assume  $\zeta_i \neq 0$ for any $i >1$, so that  (\ref{x3j0d}) holds.
We need to describe better  the vector fields $F_i$ in Proposition \ref{pop1}.
 Let $p \in   V \cap  f^{-1}(c) $.
We use  the coordinates $(s, x):=(x_1 =s, x_2, \cdots ,  x_n )$ of Lemma  \ref{threesolbx} (v).
The function $s$ is defined modulo a constant, so we set $s(p)  = 0 $ for   $p \in V \cap  f^{-1}(c)$.
Recall   $s( \theta_t  (p)) = t$.
For $i \geq 2$,  $x_i (  \theta_t (p) ) = x_i (   p )$ because
$ \frac{d}{ dt } \{ x_i (  \theta_t (p) ) \} = dx_i  (  \frac{d}{ dt }\theta_t (p)  ) =dx_i  (  E_1  )   =dx_i  ( \frac{\partial }{\partial x_1}  ) =0  $.   For $k \geq 2$ and $i \geq 1$, ${\theta_t}_* ( \frac{\partial}{\partial x_k}  ) x_i= ( \frac{\partial}{\partial x_k}  ) (x_i \circ{\theta_t})  = ( \frac{\partial}{\partial x_k}  ) (x_i )  = \delta_{ik}$, so
we get ${\theta_t}_* ( \frac{\partial}{\partial x_k}  )  =  \frac{\partial}{\partial x_k} $.

For $i \geq 2$, $F_i$ is tangent to  the level surfaces of $f$, so
we may write $F_i = \sum_{l=2}^n a_{il} (s, x)  \frac{\partial}{\partial x_l}$ for some function $  a_{il}$ of $(s, x)$. Then $F_i(p) = \sum_{l=2}^n a_{il} (0, x(p))  \frac{\partial}{\partial x_l}$.
From above we get
$F_i (\theta_t (p)) = e^{ \int_{0}^t -  \zeta_i(v) dv }  {\theta_t}_* {F_i}(p) =  e^{ \int_{0}^t   -\zeta_i(v) dv }   \sum_{l=2}^n a_{il} (0, x(p))  \frac{\partial}{\partial x_l}$. This can be rewritten in the coordinates $(s, x)$ as
\begin{equation} \label{e07}
F_i (s, x) =  e^{ \int_{0}^s   -\zeta_i(v) dv }  e_i(x).
\end{equation}
where $e_i (x) = \sum_{l=2}^n a_{il} (0, x)  \frac{\partial}{\partial x_l}$, which form a basis for the tangent space to the level surface of $f$ at the point with coordinates $(s, x)$.

As
$F_i (s, x) =  e^{ \int_{0}^s  - \zeta_i(v) dv }   e_i ( x)$,
we have
$[F_i,  F_j](s,x) = e^{ \int_0^s -(\zeta_i +\zeta_j ) (v) dv  } [e_i, e_j]      $.
Here $ [e_i, e_j]$ is tangent to  the level surfaces of $f$. We may write
$[e_i, e_j] ( x) = \sum_{l=2}^n \tilde{\gamma}_{ij}^l (x)   e_l  = \sum_{l=2}^n \tilde{\gamma}_{ij}^l ( x) e^{ \int_{0}^s   \zeta_l(v) dv }  F_l$ for a function $\tilde{\gamma}_{ij}^l (x)$ depending only on $x$. For $i, j, k \geq 2$,  we get
\begin{equation} \label{e01}
 g([F_i,  F_j] , F_k    )(s, x)
=e^{ \int_0^s -(\zeta_i +\zeta_j - \zeta_k )  (v) dv  } \cdot \tilde{\gamma}_{ij}^k ( x). \end{equation}
From  \cite[p.2]{CE}, we have
\begin{equation}\label{e02}
2g(\nabla_{F_i}F_j , F_k)= g([F_i,  F_j] , F_k    ) - g([F_i,  F_k] , F_j    ) - g([F_j,  F_k] , F_i    ). \end{equation}
%\hspace{2.6cm}\\=e^{ \int_0^t -(\zeta_i +\zeta_j - \zeta_k )  (s_o +v) dv  }\gamma_{ij}^k (s_0, x(p)) - g([F_i,  F_k] , F_j    ) - g([F_j,  F_k] , F_i    ). \hspace{1.7cm}

We shall use the equivalence class $[\cdot ]$ in (\ref{x3j0dk}).
If $[i]=[j]=[k]$,   from (\ref{e01}) and (\ref{e02}) we get $\Gamma_{ij}^k (s,x) =e^{\int_0^s -\zeta_i   (v)dv   } {\hat\gamma}_{ij}^k ( x) $, where  $\hat{\gamma}_{ij}^k (x)= \frac{1}{2} ( \tilde{\gamma}_{ij}^k ( x)  -  \tilde{\gamma}_{ik}^j ( x)  -  \tilde{\gamma}_{jk}^i  ( x)  ) $.

For $[i] \neq [2]$,  from (\ref{x3j0d}),
$e^{\int_0^s -\zeta_i   (v)dv   } = e^{-  \int_0^s (\frac{u_2^{'} }{u_2  } +  \frac{h^{'}}{ c_i + h })   (v)dv   }  =\frac{ | u_2(0 ) \cdot(c_i +h(0 ))   | } {|u_2(s) \cdot (c_i +h(s)) |}  $.  As $u_2(s) \cdot (c_i +h(s)) $ is nowhere zero and we may assume that $s$ is in an interval containing $0$, we can get $e^{\int_0^s -\zeta_i   (v)dv   }   =\frac{ u_2(0 ) \cdot(c_i +h(0 ))    } {u_2(s) \cdot (c_i +h(s)) }  $,
 so that
$\Gamma_{ij}^k =\frac{{\gamma}_{ij}^k (x)  }{ u_2(s)  (c_i +h(s)) }  $ if  $[i]=[j]=[k] \neq [2]$, where ${\gamma}_{ij}^k:=u_2 (0 )(c_i +h(0)) \hat{\gamma}_{ij}^k $.
When $[i]=[j]=[k]= [2]$, we can get $\Gamma_{ij}^k =\frac{{\gamma}_{ij}^k }{u_2 (s)}  $ similarly with   ${\gamma}_{ij}^k:=u_2 (0 ) \hat{\gamma}_{ij}^k $. We have got

\begin{lemma} \label{L10}
If $[i]=[j]=[k]$, then  $\Gamma_{ij}^k (s, x) = \frac{ {\gamma}_{ij}^k(x)}{u_2(s)( h(s)+ c_i ) }$ for $i \notin [2]$ and $\Gamma_{ij}^k = \frac{   {\gamma}_{ij}^k(x)}{u_2(s) }$ for $i \in [2]$.
\end{lemma}

\bigskip
Let $i, j$ be an integer $2 \leq i, j \leq n$ with $[i] \neq [j]$.
We use Lemma \ref{abc60byx} and  (\ref{e07u})  to compute
the following Jacobi identity;
%$0= [[F_i, F_j], F_1] + [[F_j, F_1], F_i] +  [[F_1, F_i], F_j]  \\ = [\sum_{k >1}^n (\Gamma_{ij}^k -\Gamma_{ji}^k ) F_k, F_1] + [\zeta_j F_j - \sum_{t \in [j]} \Gamma_{1j}^t F_t , F_i] -  [\zeta_i F_i  - \sum_{t \in [i]} \Gamma_{1i}^t F_t , F_j]\\=  \sum_{k >1}^n (\Gamma_{ij}^k -\Gamma_{ji}^k )(\zeta_k    -    \zeta_i - \zeta_j      ) F_k- \sum_{k >1}^nF_1 (\Gamma_{ij}^k -\Gamma_{ji}^k ) F_k\\-\sum_{k >1}^n (\Gamma_{ij}^k -\Gamma_{ji}^k ) \sum_{t \in [k] } \Gamma_{1k}^t F_t - [  \sum_{t \in [j]} \Gamma_{1j}^t F_t , F_i] + [   \sum_{t \in [i]} \Gamma_{1i}^t F_t , F_j]$
\begin{eqnarray*}
0= [[F_i, F_j], F_1] + [[F_j, F_1], F_i] +  [[F_1, F_i], F_j] \hspace{3.1cm}  \\
= [\sum_{k >1}^n (\Gamma_{ij}^k -\Gamma_{ji}^k ) F_k, F_1] + [\zeta_j F_j , F_i] -  [\zeta_i F_i , F_j] \hspace{2.6cm}\\
= \sum_{k >1}^n (\Gamma_{ij}^k -\Gamma_{ji}^k )[ F_k, F_1] - \sum_{k >1}^n F_1 (\Gamma_{ij}^k -\Gamma_{ji}^k ) F_k  -(\zeta_i + \zeta_j )[ F_i , F_j] \\
= \sum_{k >1}^n (\Gamma_{ij}^k -\Gamma_{ji}^k )(\zeta_k    -    \zeta_i - \zeta_j      ) F_k  - \sum_{k >1}^n F_1 (\Gamma_{ij}^k -\Gamma_{ji}^k ) F_k .\hspace{1.4cm}
\end{eqnarray*}
From this we obtain, for $2 \leq i, j, k \leq n$  with $[i] \neq [j]$,
\begin{equation} \label{eq5} F_1 (\Gamma_{ij}^k -\Gamma_{ji}^k )  =  (\zeta_k   -\zeta_i - \zeta_j )( \Gamma_{ij}^k -\Gamma_{ji}^k)  .
\end{equation}

We shall prove

\begin{lemma}\label{L11}
Let $[i], [j], [k]$ be pairwise different. Then the following holds.

{\rm  (i)} $\Gamma_{ij}^k = \frac{  {\gamma}_{ij}^k(x)}{u_2( h+ c_i ) }$ if  $[i], [j], [k]$ are all different from  $ [2]$.

{\rm  (ii)} $\Gamma_{ij}^k = \frac{   {\gamma}_{ij}^k(x)}{u_2 }$ for $i \in [2]$,

{\rm  (iii)}  $ \Gamma_{ij}^k=\frac{\gamma_{ij}^k(x)  }{ u_2 (h+ c_i)}   $ for $j \in [2]$ or  $k \in [2]$.

\end{lemma}

\begin{proof}
If $[i], [j], [k]$ are pairwise different,
Lemma \ref{abc60x} (ii) and (\ref{lambda06ax})  give $ \Gamma_{ji}^k = \frac{\zeta_j -  \zeta_k }{\zeta_i - \zeta_k     } \Gamma_{ij}^k$.

To prove (i);  if $[i], [j], [k], [2]$ are pairwise different, then  from (\ref{x3j0d}),
 $\frac{\zeta_i -  \zeta_j }{\zeta_i - \zeta_k     } = \frac{c_i -  c_j }{c_i - c_k     }  \frac{c_k + h }{c_j +h  }  $.
So,
$ \Gamma_{ij}^k -  \Gamma_{ji}^k  = \frac{\zeta_i -  \zeta_j }{\zeta_i - \zeta_k     } \Gamma_{ij}^k
  =  \frac{c_i -  c_j }{c_i - c_k     }  \frac{c_k + h }{c_j +h  }  \Gamma_{ij}^k
 $.

$ F_1 (\ln|\Gamma_{ij}^k -\Gamma_{ji}^k| )=   F_1 (\ln| \frac{c_i -  c_j }{c_i - c_k     }  \frac{c_k + h }{c_j +h  }  \Gamma_{ij}^k| ) =   F_1 (\ln|  h   + c_k|-  \ln| h   + c_j| + \ln |\Gamma_{ij}^k| )\\
= \zeta_k -  \zeta_j + F_1 \ln |\Gamma_{ij}^k|
  $.

Meanwhile, from (\ref{eq5}),
$ F_1 (\ln|\Gamma_{ij}^k -\Gamma_{ji}^k| ) = \zeta_k   -\zeta_i - \zeta_j
  $. By comparison we obtain
 $F_1 \ln |\Gamma_{ij}^k| =\frac{\partial }{\partial s} \ln |\Gamma_{ij}^k| = - \zeta_i =   -\frac{u_2^{'} }{u_2  }-\frac{h^{'}}{ h  + c_i}  $. Integrating this, we get
$\Gamma_{ij}^k =  \frac{  {\gamma}_{ij}^k(x) }{ u_2(h+ c_i)}, $ similarly as in the proof of Lemma \ref{L10}.

\medskip
To prove (ii);  if  $[i]=[2], [j], [k]$ are pairwise different, then
 $\frac{\zeta_2 -  \zeta_j }{\zeta_2 - \zeta_k     } =  \frac{c_k + h }{c_j +h  }  $. Following the computation in case (i), we  get $\Gamma_{ij}^k =  \frac{  {\gamma}_{ij}^k(x) }{ u_2}. $

To prove (iii);  if $[i], [j]=[2], [k]$ are pairwise different, then
 $\frac{\zeta_i -  \zeta_2 }{\zeta_i - \zeta_k     } =  \frac{c_k + h }{c_k - c_i  }  $.

$ F_1 (\ln|\Gamma_{ij}^k -\Gamma_{ji}^k| )=   F_1 (\ln|   \frac{c_k + h }{c_k - c_i }  \Gamma_{ij}^k| ) =   F_1 (\ln|  h   + c_k| + \ln |\Gamma_{ij}^k| )\\
= \zeta_k - \zeta_2+ F_1 \ln |\Gamma_{ij}^k|.
  $ From (\ref{eq5}),
$ F_1 \ln |\Gamma_{ij}^k| = -\zeta_i
=- (\ln | u_2 (h+ c_i)  |)^{'}  $.

We get $ \Gamma_{ij}^k
=\frac{\gamma_{ij}^k(x)  }{ u_2 (h+ c_i)}   $. As $ \Gamma_{ij}^k = -\Gamma_{ik}^j  $, we get (iii).

%$ \Gamma_{ji}^k = \frac{\zeta_j -  \zeta_k }{\zeta_i - \zeta_k     } \Gamma_{ij}^k= \frac{\zeta_2 -  \zeta_k }{\zeta_i - \zeta_k     } \frac{\gamma_{ij}^k  }{ u_2 (h+ c_i)}= \frac{1}{c_i - c_k     } \frac{\gamma_{ij}^k  }{ u_2 }  $.

%,  where  ${\gamma}_{ij}^k   =  \frac{c_i - c_k}{c_j - c_k}  {\gamma}_{ji}^k $  from  $ \Gamma_{ji}^k = \frac{\zeta_j -  \zeta_k }{\zeta_i - \zeta_k     } \Gamma_{ij}^k$.
\end{proof}
Now suppose $[i]\neq[j]= [k]$.
From (\ref{y03}) and (\ref{eq5}), we get   $F_1 (\Gamma_{ij}^k  )  =     -\zeta_i ( \Gamma_{ij}^k )$. (\ref{x3j0d}) gives the following

\begin{lemma} \label{L12}
Suppose $[i]\neq[j]= [k]$. Then we have

%$\Gamma_{ij}^k = e^{ \int_{0}^t -  \zeta_i(v) dv }   {\gamma}_{ij}^k  $, where the function ${\gamma}_{ij}^k $ does not depend on the variable $s$.

 $\Gamma_{ij}^k = \frac{ {\gamma}_{ij}^k(x)}{ u_2(h+ c_i ) }$ if $ [i] \neq [2]$,  and $\Gamma_{ij}^k =\frac{ {\gamma}_{ij}^k(x)}{u_2 }$ if $ [i] = [2]$.
\end{lemma}

\section{Analysis of the Ricci tensor and the gradient Ricci soliton equation}

In this section we study   the Ricci tensor  when a GRS with harmonic Weyl curvature has at least two distinct $\lambda_i$'s among $i \geq 2$ in a refined adapted frame field $F_i$.  We still assume  $\zeta_i \neq 0$ for any $i >1$.
Set
\begin{eqnarray} \label{u1}
X_i = \sum_{1<j\leq n, j \neq i} \{   F_i \Gamma_{jj}^i  + F_j \Gamma_{ii}^j -(\Gamma_{ii}^j)^2  -(\Gamma_{jj}^i)^2 \} \hspace{1.7cm}\\ \ \ +  \sum_{1<j\leq n, j \neq i} \sum_{k \neq 1, i, j } \{ - \Gamma_{jj}^k \Gamma_{ii}^k
   -\Gamma_{ij}^k\Gamma_{jk}^i  -( \Gamma_{ij}^k-\Gamma_{ji}^k) \Gamma_{kj}^i \} \nonumber
   \end{eqnarray}
so that the formula of $R_{ii}$ in $F_i$ in Lemma \ref{77bx01} becomes
 \begin{eqnarray}
 R_{ii}  = -\zeta_i^{'}   -\zeta_i^2   -\zeta_i (\sum_{1<j\leq n, j \neq i} \zeta_j) +X_i. \label{rii}
\end{eqnarray}
Then $X_i$ is independent of the choice of $i$ in $[i]$.
We also recall $ R_{ii} = \lambda - \zeta_i f^{'}  $.
For $i, s$ with $[i] \neq [s],$  by (\ref{x2}),
$ (\zeta_s -\zeta_i)f^{'}= R_{ii}  -   R_{ss}= \zeta_i^2 - \zeta_s^2  +  (\zeta_s -\zeta_i) (\sum_{1<j\leq n}   \zeta_j )     +X_i - X_s.$ So,
 \begin{eqnarray}
f^{'}=- (\zeta_i+\zeta_s) +   \sum_{1<j\leq n} \zeta_j     + \frac{X_i - X_s }{ (\zeta_s -\zeta_i) },   \ \ \ \     {\rm for } \ \ [i] \neq [s]. \label{riis}
\end{eqnarray}
We set $Z_{is}: =- (\zeta_i+\zeta_s) +   \sum_{1<j\leq n} \zeta_j     + \frac{X_i - X_s }{ (\zeta_s -\zeta_i) } $, which should be independent of the classes $[i], [s]$ as  the LHS of (\ref{riis}) is.

We shall treat the  terms of $X_i$ in (\ref{u1}).
For  $1< i \notin [2]$,
from  (\ref{x3j0d}) and (\ref{e07}), we get
$F_i  = \frac{  u_2(0 ) \cdot(c_i +h(0 ))    } {u_2(s) \cdot (c_i +h(s)) }  e_i $. Set $a_i = u_2(0 ) \cdot(c_i +h(0 ))   $ which is independent of $i \in [i]$.  So,
we write
$F_i = \frac{a_i }{u_2(h+ c_i)}e_i$.
 For  $i \in [2]$,  similarly  $F_i= \frac{a_2 }{u_2}e_i$  for a constant $a_2$ which is independent of $i$.  We write;
  \begin{eqnarray} \label{x3j0d5}
 \begin{cases}
 \  F_i = \frac{a_i }{u_2(h+ c_i)}e_i  \ \  \ {\rm when}  \ \ 1< i \notin [2],\\
 \   F_i= \frac{a_2 }{u_2}e_i  \ \  \ \ \ \  \ \ {\rm when}  \ \ [i]= [2].
 \end{cases}
\end{eqnarray}

 Using $h^{'}  = \frac{1}{u_2^2}$, (\ref{y03}), (\ref{x3j0d5}) and  Lemma \ref{L10}, we get for  $1< i \notin [2]$
 \begin{eqnarray} \label{aq01}
\sum_{1<j\leq n, j \neq i} \{  F_i \Gamma_{jj}^i  + F_j \Gamma_{ii}^j -(\Gamma_{ii}^j)^2  -(\Gamma_{jj}^i)^2 \}\hspace{2cm}  \\
% = h^{'} \sum_{j \in [i], j \neq i} \{\frac{a_i }{h+ c_i} \frac{a_j }{h+ c_j}e_i {\gamma}_{jj}^i  - \frac{a_j }{h+ c_j} \frac{a_i }{h+ c_i} e_j {\gamma}_{ii}^j - \frac{a_i^2 }{(h+ c_i)^2} ({\gamma}_{ii}^j)^2- \frac{a_j^2 }{(h+ c_j)^2} ({\gamma}_{jj}^i)^2\}  \nonumber \\
 =\frac{h^{'}}{(h+ c_i)^2} \sum_{j \in [i], j \neq i} \{ a_i  e_i {\gamma}_{jj}^i  - a_j e_j {\gamma}_{ii}^j -  ({\gamma}_{ii}^j)^2-  ({\gamma}_{jj}^i)^2\}, \hspace{0cm} \nonumber
\end{eqnarray}
and for $i \in [2]$,
 \begin{eqnarray} \label{aq01b}
\sum_{1<j\leq n, j \neq i} \{  F_i \Gamma_{jj}^i  + F_j \Gamma_{ii}^j -(\Gamma_{ii}^j)^2  -(\Gamma_{jj}^i)^2 \}\hspace{2cm}  \\
% = h^{'} \sum_{j \in [i], j \neq i} \{\frac{a_i }{h+ c_i} \frac{a_j }{h+ c_j}e_i {\gamma}_{jj}^i  - \frac{a_j }{h+ c_j} \frac{a_i }{h+ c_i} e_j {\gamma}_{ii}^j - \frac{a_i^2 }{(h+ c_i)^2} ({\gamma}_{ii}^j)^2- \frac{a_j^2 }{(h+ c_j)^2} ({\gamma}_{jj}^i)^2\}  \nonumber \\
 =h^{'}  \sum_{j \in [i], j \neq i} \{  a_2e_i {\gamma}_{jj}^i  -  a_2e_j {\gamma}_{ii}^j -  ({\gamma}_{ii}^j)^2-  ({\gamma}_{jj}^i)^2\} . \hspace{0cm} \nonumber
\end{eqnarray}

For  $1< i \notin [2]$, using (\ref{y03}) and  Lemma \ref{L10} again,
 \begin{eqnarray}\label{aq02}
\sum_{1<j\leq n, j \neq i}  \{ \sum_{k \neq 1, i, j } - \Gamma_{jj}^k \Gamma_{ii}^k \}=
\frac{ -h^{'}}{(h+ c_i)^2}  \sum_{j \in [i], j \neq i} \{\sum_{k \neq 1, i, j , k \in [i]}  {\gamma}_{jj}^k{\gamma}_{ii}^k\},
\end{eqnarray}
and for $i \in [2]$,
 \begin{eqnarray}\label{aq02b}
\sum_{1<j\leq n, j \neq i} \{ \sum_{k \neq 1, i, j } - \Gamma_{jj}^k \Gamma_{ii}^k \}=  -h^{'} \sum_{j \in [i], j \neq i} \{\sum_{k \neq 1, i, j , k \in [i]}{\gamma}_{jj}^k {\gamma}_{ii}^k\}.
\end{eqnarray}
Set $\alpha_{ij}^k := -\Gamma_{ij}^k\Gamma_{jk}^i  -( \Gamma_{ij}^k-\Gamma_{ji}^k) \Gamma_{kj}^i$.
When $[i], [j], [k]$ are pairwise different,  $\alpha_{ij}^k  =2\Gamma_{ji}^k \Gamma_{ij}^k  $  from Lemma \ref{abc60x} (ii) and (\ref{x1z}). When exactly two of $[i], [j], [k]$ are equal,  $\alpha_{ij}^k=0$
from  (\ref{y03}).  Using these and Lemma \ref{L10} and Lemma \ref{L11}, we compute
for $1< i \notin [2]$,
 \begin{eqnarray}\label{aq03}
\ \ \sum_{1<j \leq n, j \neq i}   \sum_{k \neq 1, i, j } \alpha_{ij}^k \hspace{7cm} \\
  = \sum_{j \in [i], j \neq i} \{\sum_{k \neq 1, i, j, k \in [i]}\alpha_{ij}^k \}
   +\sum_{1<j \notin [i]}^n \{\sum_{k \neq 1, k \notin [i],  k \notin [j] }
   2\Gamma_{ji}^k \Gamma_{ij}^k \}  \nonumber  \hspace{2cm} \\
     = h^{'} \sum_{j \in [i], j \neq i} \{\sum_{k \neq 1, i, j, k \in [i]}
   \frac{-{\gamma}_{ij}^k  {\gamma}_{jk}^i -{\gamma}_{ij}^k{\gamma}_{kj}^i +{\gamma}_{ji}^k{\gamma}_{kj}^i }{(h+ c_i)^2}  \}  \nonumber  \hspace{3cm} \\
   +h^{'}\sum_{j \in [2]} \{\sum_{k \neq 1, k \notin [i],  k \notin [j] }
 2  \frac{ {\gamma}_{ji}^k(x){\gamma}_{ij}^k(x)}{ h+ c_i  } \}
   +h^{'}\sum_{1<j \notin [i], [2]}^n \{\sum_{k \neq 1, k \notin [i],  k \notin [j] }
   2 \frac{ {\gamma}_{ji}^k(x)}{ h+ c_j  } \frac{ {\gamma}_{ij}^k(x)}{ h+ c_i  } \}. \nonumber  \hspace{0cm}
\end{eqnarray}

 \bigskip
For $i \in [2]$,
 \begin{eqnarray}\label{aq03b}
\ \ \ \ \  \sum_{1<j \leq n, j \neq i}   \sum_{k \neq 1, i, j }\alpha_{ij}^k   \hspace{7cm} \\
  = \sum_{j \in [i], j \neq i} \{\sum_{k \neq 1, i, j, k \in [i]} \alpha_{ij}^k \}
   +\sum_{1<j \notin [i]}^n \{\sum_{k \neq 1, k \notin [i],  k \notin [j] }
   2\Gamma_{ji}^k \Gamma_{ij}^k \}   \nonumber  \hspace{0cm} \\
     = h^{'} \sum_{j \in [2], j \neq i} \{\sum_{k \neq 1, i, j, k \in [2]}
   (-{\gamma}_{ij}^k  {\gamma}_{jk}^i -{\gamma}_{ij}^k{\gamma}_{kj}^i +{\gamma}_{ji}^k{\gamma}_{kj}^i ) \}  \nonumber  \hspace{0.7cm} \\
   +h^{'}\sum_{1<j \notin  [2]}^n \{\sum_{k \neq 1, k \notin [2],  k \notin [j] }
   2 \frac{ {\gamma}_{ij}^k(x) {\gamma}_{ji}^k(x)}{ h+ c_j  }  \}
   . \nonumber  \hspace{2cm}
\end{eqnarray}

%Recall (\ref{x5})\begin{eqnarray*}\lambda_i  = \lambda - \zeta_i f^{'}  =\lambda - \{ \frac{u_2^{'} }{u_2  } +  \frac{1}{(u_2)^2 (c_i +\int_{s_0}^s \frac{ds}{ (u_2)^2 }) }\} f^{'} .\end{eqnarray*}

%With $h=\int_{s_0}^s \frac{ds}{ (u_2)^2 }$

\bigskip
For  $1< i \notin [2]$, from  (\ref{u1}), (\ref{rii}), (\ref{aq01}),   (\ref{aq02}), (\ref{aq03}) and $\zeta_i^{'}  + \zeta_i^2 = \frac{u_2^{''}}{u_2}$,
the formula $ \lambda - \zeta_i f^{'} = R_{ii} $  gives
 \begin{eqnarray} \label{aq7}
\lambda - \{ \frac{u_2^{'} }{u_2  } +  \frac{h'}{ (c_i +h) }\} f^{'}
   = -\frac{u_2^{''} }{u_2 }  \hspace{5cm}   \\
  - \sum_{1<j \notin [2], j \neq i}( \frac{u_2^{'} }{u_2  } +  \frac{h^{'}}{ c_i +h })  ( \frac{u_2^{'} }{u_2  } +  \frac{h^{'}}{{ c_j} +h }  )
    - \sum_{j \in [2]} ( \frac{u_2^{'} }{u_2  } +  \frac{h^{'}}{ c_i +h })  ( \frac{u_2^{'} }{u_2  }   )  + X_i. \nonumber
\end{eqnarray}
where
 \begin{eqnarray} \label{aq7s}
\ \ \ \ \ X_i = h^{'} \{ \frac{ A_i}{(h+ c_i)^2}
   + \frac{ 2 }{ h+ c_i  } \sum_{j \in [2]} \beta_{ij}
   +2   \sum_{1<j \notin [i], [2]}^n
   \frac{   \beta_{ij} }{(h+ c_i)(h+ c_j)} \}.
\end{eqnarray}
In the above,
$A_i :=  \sum_{j \in [i], j \neq i} \{  a_i e_i {\gamma}_{jj}^i  - a_j e_j {\gamma}_{ii}^j - ({\gamma}_{ii}^j)^2- ({\gamma}_{jj}^i)^2
   - \sum_{k \neq 1, i, j , k \in [i]} {\gamma}_{jj}^k {\gamma}_{ii}^k \\
+ \sum_{k \neq 1, i, j, k \in [i]}
  ( -{\gamma}_{ij}^k  {\gamma}_{jk}^i -{\gamma}_{ij}^k{\gamma}_{kj}^i +{\gamma}_{ji}^k{\gamma}_{kj}^i )  \},
  $   and

\medskip
    $\beta_{ij} :=  \sum_{k \neq 1, k \notin [i],  k \notin [j] }
   {\gamma}_{ji}^k(x) {\gamma}_{ij}^k(x).$ Note that $A_i$ and $\beta_{ij} $ depend on $x$, not on $s$,    and that $\beta_{ij} = \beta_{ji}$.

For $i \in  [2]$,
 from   (\ref{u1}), (\ref{rii}), (\ref{aq01b}),   (\ref{aq02b}) and  (\ref{aq03b}),  the formula $ \lambda - \zeta_i f^{'} = R_{ii} $ gives
 \begin{eqnarray} \label{aq7b}
\ \ \ \ \lambda -  \frac{u_2^{'} }{u_2  }f^{'}
   = -\frac{u_2^{''} }{u_2 }
  - \sum_{1<j \notin [2]}   \frac{u_2^{'} }{u_2  }  ( \frac{u_2^{'} }{u_2  } +  \frac{h^{'}}{{ c_j} +h }  )     - \sum_{j \in [2], j \neq i}  \frac{u_2^{'} }{u_2  }   \frac{u_2^{'} }{u_2  }    + X_i,
\end{eqnarray}
where
 \begin{eqnarray} \label{aq7bs}
X_i  =  h^{'}(  A_i
   +\sum_{1<j \notin  [2]}^n
   2 \frac{ \beta_{ij} }{ h+ c_j  }).
\end{eqnarray}

At this point we need the following elementary lemma.
%;  to prove the lemma, one can ?? extend    and multiply the equation by some polynomials with factors $(x+ c_i)$ and compare LHS and RHS.
\begin{lemma} \label{L41}
Let $k_0$, $k_1$ and $n_0$ be natural numbers and $c_1, \cdots , c_{n_0}$ be pairwise different real numbers. Let $G(x)$ be a function defined on an open interval $I \subset \mathbb{R}$ by  $G(x) := \sum_{s=0}^{k_0} \beta_{s} x^s  + \sum_{i=1}^{n_0}\sum_{t=1}^{k_1}  \frac{\alpha_{i, t}}{(x+ c_i)^t} $, for  constants  $\beta_s$ and  $\alpha_{i. t} $.

Suppose that $G(x) = 0$ on $I$.
 Then   $\beta_s= \alpha_{i. t}=0 $ for any $0 \leq s \leq k_0$ and any  $ 1  \leq i \leq n_0, \ 1 \leq t \leq k_1 $.
\end{lemma}
\begin{proof}
Define a function $\hat{G}(x) :=  (\sum_{s=0}^{k_0} \beta_{s} x^s  + \sum_{i=1}^{n_0}\sum_{t=1}^{k_1}  \frac{\alpha_{i, t}}{(x+ c_i)^t})\cdot \Pi_{i=1}^{n_0} (x+ c_i)^{k_1} $  on $\mathbb{R}$.
 As $G=0$ on $I$,  $\hat{G} =0 $ on $I$. The function  $\hat{G} (x)  $ is real analytic on $\mathbb{R}$, so $\hat{G} =0 $ on    $\mathbb{R}$. For each $1 \leq j \leq n_0$, we get   $0=\hat{G} (-c_j) = \alpha_{j, k_1} \cdot  \Pi_{i=1, i \neq j}^{n_0} (-c_j + c_i)^{k_1}$. So,  $ \alpha_{j, k_1} =0$.
 Continuing inductively, we can show  $ \alpha_{i, t}=0$ for any $ 1  \leq i \leq n_0$ and $ 1 \leq t \leq k_1 $. Then  $\beta_{s}=0$ for  any  $0 \leq s \leq k_0$.
\end{proof}

In this paper we often treat functions in the form of $G(x)$ in Lemma \ref{L41}.
We shall say that  $\frac{1}{(x+ c_i)^t}$ has degree $-t$.    This handy definition of degree is based on Lemma \ref{L41}. Note that $\frac{1}{x(x+1)}$ does not have degree $-2$ but degree $-1$,   because  $\frac{1}{x(x+1) }=  \frac{1}{x}  - \frac{1}{x+1} $.

\bigskip
Recall the coordinates  of Lemma \ref{threesolbx} (v).
From (\ref{x3j0da2}), $h^{'} \neq 0$.
Divide (\ref{aq7b})  by $h^{'}$ and
take derivative $\frac{\partial  }{\partial x_k},  2 \leq k \leq n$
to get
$0= \frac{\partial A_i }{\partial x_k}
   +\sum_{1< j \notin  [2]}^n
    \frac{2 }{ h+ c_j  } \frac{\partial  \beta_{ij}}{\partial x_k}$ for $i \in [2]$.
 By Lemma \ref{L41}, we get $\frac{\partial A_i }{\partial x_k} =0$ and for each class $[l_0] \neq [2]$,
     $\sum_{j \in  [l_0] } \frac{\partial  \beta_{ij}}{\partial x_k} =0 $. So. $A_i$ and
 $\sum_{j \in  [l_0] }  \beta_{ij} $ are constants.
 We claim that $A_i$ and
 $\sum_{j \in  [l_0] }  \beta_{ij} $ are independent of the choice of $i$ in $[i]=[2]$.
 Assume that the $\lambda_2$ eigenspace has bigger than one dimension.
As (\ref{aq7b}) holds for any $i_1, i_2 \in [2]$, subtracting one by the other we get
$0=( A_{i_1}  -  A_{i_2}  )
   +\sum_{1<j \notin  [2]}^n
   2 \frac{1 }{ h+ c_j  }   ( \beta_{i_1 j} -\beta_{i_2 j}   ) $.
Again by Lemma \ref{L41}, $A_{i_1} - A_{i_2} =0$
and  $\sum_{j \in  [l_0] } \beta_{i_1 j} -\sum_{j \in  [l_0] } \beta_{i_2 j} =0. $
So, the claim is proved.

\bigskip
Similarly, for $1< i \notin [2]$, divide (\ref{aq7})  by $h^{'}$ and
take derivative $\frac{\partial  }{\partial x_k},  2 \leq k \leq n$
to get
 \begin{eqnarray*}
0=
  \frac{1}{(h+ c_i)^2}   \frac{\partial A_i }{\partial x_k}
   + \frac{ 2 }{ h+ c_i  }\{ \sum_{j \in [2]} \frac{\partial  \beta_{ij}}{\partial x_k}
   +  \sum_{1<j \notin [i], [2]}
   \frac{1 }{(c_j - c_i )} \frac{\partial  \beta_{ij}}{\partial x_k} \} \\
  -  \sum_{1<j \notin [i], [2]} \frac{2 }{h+ c_j}
   \frac{  \frac{\partial  \beta_{ij}}{\partial x_k} }{(c_j - c_i )}  . \hspace{4.5cm}
\end{eqnarray*}

We then get three equalities by Lemma \ref{L41};

(i) $\frac{\partial A_i }{\partial x_k} =0$,

 (ii)   $ \sum_{j \in [2]} \frac{\partial  \beta_{ij}}{\partial x_k}
   +  \sum_{1<j \notin [i], [2]}
   \frac{ 1 }{(c_j - c_i )} \frac{\partial  \beta_{ij}}{\partial x_k}   =0$,

(iii) $\sum_{j \in [l_0] } \frac{\partial  \beta_{ij}}{\partial x_k} =0 $ for $[l_0] \neq  [i], [2]$.

\medskip
We use  similar argument to dealing with (\ref{aq7b}).
For $i \notin [2]$,   $A_i$ is a constant independent of the choice of $i$ in $[i]$.
Also from (iii)  $\sum_{j \in [l_0] }  \beta_{ij}$ with $[l_0] \neq  [i], [2]$  is a constant independent of the choice of $i$ in $[i]$.
This gives  $ \sum_{1< j \notin [i], [2]}
   \frac{ 1 }{(c_j - c_i )} \frac{\partial  \beta_{ij}}{\partial x_k}   =0$ in (ii). So, by (ii) $\sum_{j \in [2]}  \beta_{ij}$   is a constant independent of the choice of $i$ in $[i]$.   We proved

 % Now  for $i \geq 2$, $X_i$ can be checked to be independent of $i$ in $[i]$.

\begin{lemma} \label{cons}
 For $i \geq 2$ and  $[l_0] \neq  [i]$,  $A_i$ and $\sum_{j \in  [l_0] }  \beta_{ij} $  are  constants independent of the choice of $i$ in $[i]$.
\end{lemma}

%From the last one, we get  $\sum_{j \notin [i], [2]}   \frac{ a_j }{(c_j - c_i )} \frac{\partial  \beta_{ij}}{\partial x_k} =0.$

%So, $ \sum_{j \in [2]}^n \frac{\partial  \beta_{ij}}{\partial x_k} =0.$

% \begin{eqnarray} \label{aq7n2}u_2^2\lambda -u_2^2 \{ \frac{u_2^{'} }{u_2  } +  \frac{1}{(u_2)^2 (c_i +\int_{s_0}^s \frac{ds}{ (u_2)^2 }) }\} f^{'} \hspace{1cm}\\   = -u_2^2\frac{u_2^{''} }{u_2 }   + \sum_{1<j\leq n, j \neq i}  u_2^2\{  -( \frac{u_2^{'} }{u_2  } +  \frac{h^{'}}{ \hat{ c_i} +h })  ( \frac{u_2^{'} }{u_2  } +  \frac{h^{'}}{ \hat{ c_j} +h }  )  \nonumber \\   +\frac{A_i }{(h+ c_i)^2}   +\sum_{j \notin [i]}^n \{\sum_{k \notin [i],  k \notin [j] }   2\frac{ {\gamma}_{ji}^k(x)}{ h+ c_j  } \frac{ {\gamma}_{ij}^k(x)}{ h+ c_i  } \}. \hspace{1cm} \nonumber\end{eqnarray}

% $\zeta_t-\zeta_s     + h^{'} \frac{\frac{A_i a_i^2 }{(h+ c_i)^2} - \frac{A_s a_s^2}{(h+ c_s)^2} +B_i - B_s }{ (\zeta_s -\zeta_i) }=    h^{'}  \frac{\frac{A_ia_i^2  }{(h+ c_i)^2} - \frac{A_t a_t^2 }{(h+ c_t)^2} + B_i - B_t }{ (\zeta_t -\zeta_i) } .$

%\medskip    $\beta_{ij} :=  \sum_{k \notin [i],  k \notin [j] }  {\gamma}_{ji}^k(x) {\gamma}_{ij}^k(x).$

\section{At least three distinct $\lambda_i$'s  among $i \geq 2 $ }

Let $(M^n,g,f)$, $n \geq 4$, be an $n$-dimensional gradient Ricci soliton with harmonic Weyl curvature. In this section we further assume that $(M^n,g,f)$ has  at least three distinct $\lambda_i$'s  among $i \geq 2 $  in a refined adapted frame field $F_i$ and shall draw a contradiction so that we obtain Proposition \ref{pro3}.

By Lemma \ref{gg},   $\zeta_i \neq 0$ for any $i >1$.
We denote by $n_i$ the dimension of the $\lambda_i$-eigenspace.
Recall that $Z_{is}$, the right hand side of (\ref{riis}) does not depend on the classes $[i], [s]$ while $[i] \neq [s]$. If  $[t], [i], [s]$ are pairwise different, the equality $Z_{is}- Z_{it}=0$ yields

% $- (\zeta_i+\zeta_s)     + \frac{X_i - X_s }{ (\zeta_s -\zeta_i) }=- (\zeta_i+\zeta_t)     + \frac{X_i - X_t}{ (\zeta_t -\zeta_i) }. $ Then,

 \begin{eqnarray} \label{aq01n}
\zeta_t-\zeta_s     + \frac{X_i - X_s}{ (\zeta_s -\zeta_i) }
=      \frac{X_i - X_t}{ (\zeta_t -\zeta_i) }.
\end{eqnarray}

First, we use $ Z_{2s}- Z_{2t}=0$ for  $t, s \notin [2]$ with $[t] \neq [s]$.
From  (\ref{x3j0d}),  we get
$\zeta_t-\zeta_s = \frac{c_s -{ c_t} }{ ({ c_t} +h ) (c_s +h) } h^{'}$
and  $ \zeta_t-\zeta_2 = \frac{h^{'} }{ ({ c_t} +h )  }$. Putting $i=2$ into (\ref{aq01n}) and  multiplying (\ref{aq01n}) with  $ ( c_s +h ) (c_t +h) $,  we get
 \begin{eqnarray} \label{aq01n5}
  \ \ \ \ \ \  \ \ \ (c_s - c_t )  h^{'}  +\frac{( c_s +h ) (c_t +h)( X_2 - X_s)}{(\zeta_s -\zeta_2)}
  =\frac{ ( c_s +h ) (c_t +h)( X_2 - X_t)}{(\zeta_t -\zeta_2)}.
\end{eqnarray}

From (\ref{aq7s}) and (\ref{aq7bs}), we get
 \begin{eqnarray} \label{aq01n5j}
\frac{( X_2 - X_s)}{(\zeta_s -\zeta_2)
}  =(h + c_s )  (A_2   + Y_2- \frac{A_s  }{(h+ c_s)^2}- Y_s).
\end{eqnarray}
where $ Y_s =  \frac{2 }{ h+ c_s  } \sum_{j \in [2]} \beta_{sj}
   +2  \sum_{1< j \notin [s], [2]}
   \frac{   \beta_{sj} }{(h+ c_s )(h+ c_j  ) }   $  for $ s \notin [2]$, and   $Y_2  =  \sum_{1< j \notin  [2]}^n
   2 \frac{ \beta_{2j} }{ h+ c_j  }$.
Putting (\ref{aq01n5j})   into  (\ref{aq01n5}), we get
 \begin{eqnarray} \label{aq01n5jhg}
\ \ \ \ (c_s - c_t )  h^{'} =  -( c_s +h ) (c_t +h)(h + c_s )  (A_2   + Y_2- \frac{A_s  }{(h+ c_s)^2}- Y_s) \\
  +( c_s +h ) (c_t +h)(h + c_t )  (A_2   + Y_2- \frac{A_t  }{(h+ c_t)^2}- Y_t). \nonumber
\end{eqnarray}
We use Lemma \ref{cons} to write the right hand side of (\ref{aq01n5jhg})  in the form of the function  $G(h)$ in Lemma \ref{L41}.
Then the $h^3$ terms cancel out and
(\ref{aq01n5jhg}) yields
  \begin{eqnarray} \label{aq52}
 h^{'}= F(h)
:=   ch^2 +ah + b + \sum_{1<l \notin [2]} \frac{d_l}{h+ c_l} ,
\end{eqnarray}
 for some constants $a, b, c, d_l$;
here we may assume  $d_l$ to be independent  of $l \in [l] $ because $ c_l$ are so.
%Compute $d_l$. We only consider degree $-1$ term below$  (\hat{ c_s} -\hat{ c_t})  h^{'}\sim   (\hat{ c_s} +h) (\hat{ c_t} +h) \{ ( -2a_s  \sum_{j \notin [s], [2]}^n    a_j  \beta_{sj}  \frac{1}{h+ c_j}   \frac{1}{(c_j - c_s )} ) (h + c_s)- ( -2a_t  \sum_{j \notin [t], [2]}^n    a_j  \beta_{tj}  \frac{1}{h+ c_j}   \frac{1}{(c_j - c_t )} )(h  + c_t)+ (c_t -c_s) \sum_{j \notin  [2]}^n   2a_2 a_j \frac{ \beta_{2j} }{ h+ c_j  } \}.$ $\frac{({ c_t} +h)}{h+ c_j}    (h + c_s)^2 \sim  \frac{ (c_t - c_j)(c_s - c_j)^2}{h+ c_j } $. So, $    h^{'}\sim    \{  2  \sum_{j \notin [s], [2]}^n    a_j ( a_s \beta_{sj} - a_t \beta_{tj}  )  \frac{(c_t - c_j)(c_s - c_j)}{(\hat{ c_s} -\hat{ c_t})(h+ c_j)}- \sum_{j \notin  [2]}^n   2a_2 a_j \frac{ \beta_{2j} }{ h+ c_j  } \}.$ $  \sum_{s \neq t \notin [2]}  h^{'}\sim  \sum_{s \neq t \notin [2]}  \{  2  \sum_{j \notin [s],[t],  [2]}^n    a_j ( a_s \beta_{sj} - a_t \beta_{tj}  )  \frac{(c_t - c_j)(c_s - c_j)}{(\hat{ c_s} -\hat{ c_t})(h+ c_j)}- \sum_{j \notin  [2]}^n   2a_2 a_j \frac{ \beta_{2j} }{ h+ c_j  } \}.$
Take derivative $\frac{d}{ds}$ to get $    h^{''} = \frac{dF}{dh} h^{'}
=   (2ch + a - \sum_{1<l \notin [2]} \frac{d_l}{(h+ c_l)^2})h^{'}  .$

As $h^{'} = \frac{1}{u_2^2}$, we have $ \zeta_2= \frac{u_2^{'}}{u_2}= - \frac{h^{''} }{2h^{'}  }  = - \frac{1 }{2  } \frac{dF}{dh}$. We get
  \begin{eqnarray} \label{aq0z1}
\zeta_2  =     - ch - \frac{a}{2  }+ \sum_{1<l \notin [2]} \frac{d_l}{2 (h+ c_l)^2}.
\end{eqnarray}

Take derivative $\frac{d}{ds}$ to get $ \zeta_2^{'}= - h^{'}  (c + \sum_{1< l \notin [2]} \frac{d_l}{(h+ c_l)^3}) =-( ch^2 +ah + b + \sum_{1<l \notin [2]} \frac{d_l}{h+ c_l} )(c + \sum_{1<l \notin [2]} \frac{d_l}{(h+ c_l)^3}),
  $ we compute
  %has   terms of  degree $2, 1,  0, -1,  -2, -3,  -4$.
%$ \zeta_2^2$ has   terms of  degree $2, 1,  0, -1,  -2,  -4$.
%$\zeta_2^{'} +  \zeta_2^2  = -( ch^2 +ah + b + \sum_{l \notin [2]} \frac{d_l}{h+ c_l} )(c + \sum_{l \notin [2]} \frac{d_l}{(h+ c_l)^3}) +   (ch + \frac{1 }{2  }a - \sum_{l \notin [2]} \frac{d_l}{2(h+ c_l)^2})^2\\
%= -( c^2h^2 +cah + cb + \sum_{l \notin [2]} \frac{cd_l}{h+ c_l} ) -( ch^2 +ah + b + \sum_{l \notin [2]} \frac{d_l}{h+ c_l} )( \sum_{l \notin [2]} \frac{d_l}{(h+ c_l)^3})+c^2 h^2 + cah + \frac{a^2}{4} - ch \sum_{l \notin [2]} \frac{d_l}{(h+ c_l)^2} -a \sum_{l \notin [2]} \frac{d_l}{2(h+ c_l)^2}  +  ( \sum_{[l] \neq [2]} \frac{d_l}{2(h+ c_l)^2})^2\\=\frac{a^2}{4} -  cb - \sum_{l \notin [2]} \frac{cd_l}{h+ c_l}  -( ch^2 +ah + b + \sum_{l \notin [2]} \frac{d_l}{h+ c_l} )( \sum_{l \notin [2]} \frac{d_l}{(h+ c_l)^3}) - ch \sum_{l \notin [2]} \frac{d_l}{(h+ c_l)^2} -a \sum_{l \notin [2]} \frac{d_l}{2(h+ c_l)^2}  +  ( \sum_{l \notin [2]} \frac{d_l}{2(h+ c_l)^2})^2$
 \begin{eqnarray} \label{aq0z}
\zeta_2^{'} +  \zeta_2^2  =\frac{a^2}{4} -  cb \hspace{6.5cm}   \\
- \sum_{1< l \notin [2]} \frac{cd_l}{h+ c_l}  -( ch^2 +ah + b + \sum_{1<l \notin [2]} \frac{d_l}{h+ c_l} )( \sum_{1<l \notin [2]} \frac{d_l}{(h+ c_l)^3}) \nonumber  \\
 - ch \sum_{1<l \notin [2]} \frac{d_l}{(h+ c_l)^2} -a \sum_{1<l \notin [2]} \frac{d_l}{2(h+ c_l)^2}  +  ( \sum_{1<l \notin [2]} \frac{d_l}{2(h+ c_l)^2})^2. \ \nonumber
\end{eqnarray}

For $1< i \notin [2]$,
$ \zeta_i - \zeta_2 = \frac{h^{'}}{ { c_i} +h } =  \frac{ ch^2 + ah + b + \sum_{1<l \notin [2]} \frac{d_l}{h+ c_l} }{ { c_i} +h }.$  Using   $\frac{1}{(h+ c_l) (h+ c_i) } = \frac{1}{c_i - c_l}\{ \frac{1}{(h+ c_l) }   - \frac{1}{ (h+ c_i) } \}  $ for $c_i \neq c_l$, we get

 \begin{eqnarray} \label{aq0z3h1}
\zeta_i - \zeta_2 =ch + a -c c_i
+\frac{  cc_i^2 -ac_i + b  }{ { c_i} +h } \hspace{3.3cm} \\
 +  \sum_{1<l \notin [2], [i] }  \frac{d_l}{c_i - c_l}\{ \frac{1}{(h+ c_l) }   - \frac{1}{ (h+ c_i) } \}
 +\frac{n_i d_i}{(h+ c_i)^2}. \nonumber
\end{eqnarray}

%From (\ref{aq0z1}) and  (\ref{aq0z3}),$\zeta_i - \zeta_2 =ch + a -c c_i+\frac{  cc_i^2 -ac_i + b  }{ { c_i} +h }  +  \sum_{l \notin [2], [i] }  \frac{d_l}{c_i - c_l}\{ \frac{1}{(h+ c_l) }   - \frac{1}{ (h+ c_i) } \} \ \ \\ +  \sum_{l \in [i]}  \frac{3d_l}{2(h+ c_l)^2} +  \sum_{l \notin [i], [2]}  \frac{d_l}{2(h+ c_l)^2}- \sum_{l \notin [2]} \frac{d_l}{2 (h+ c_l)^2}.$Note that $ \sum_{l \in [i]}  \frac{3d_l}{2(h+ c_l)^2} +  \sum_{l \notin [i], [2]}  \frac{d_l}{2(h+ c_l)^2}- \sum_{l \notin [2]} \frac{d_l}{2 (h+ c_l)^2}\\=\sum_{l \in [i]}  \frac{3d_l}{2(h+ c_l)^2}  - \sum_{l \in  [i]} \frac{d_i}{2(h+ c_l)^2} =\frac{n_i d_i}{(h+ c_i)^2} .$We have got

We use (\ref{aq7s}) and $\zeta_i-\zeta_s = \frac{{ c_s} -{ c_i} }{ ({ c_i} +h ) ({ c_s} +h) } h^{'}$ for $s \notin [2]$ and $ s \notin [i]$, to write
 \begin{eqnarray} \label{aq0z37}
\frac{X_i  }{ (\zeta_s -\zeta_i) } =    \frac{(c_s +h )  }{ (c_i -c_s) }\{ \frac{ A_i}{(h+ c_i)}   + 2 \sum_{j \in [2]} \beta_{ij} +2   \sum_{1< j \notin [i], [2]}   \frac{   \beta_{ij} }{h+ c_j } \}. \hspace{0cm}
\end{eqnarray}

% \begin{eqnarray} \label{aq0z37}\frac{X_i - X_s }{ (\zeta_s -\zeta_i) } =   \hspace{8.6cm}\\\frac{(c_s +h ) ( c_i +h) }{ (c_i -c_s) }\{ \frac{ A_i}{(h+ c_i)^2}   + \frac{ 2\sum_{j \in [2]}^n \beta_{ij} }{ h+ c_i  } +2   \sum_{j \notin [i], [2]}   \frac{   \beta_{ij} }{(c_j - c_i )} ( \frac{1}{h+ c_i} - \frac{1}{h+ c_j}   ) \nonumber \\  -  \frac{ A_s}{(h+ c_s)^2} - \frac{ 2\sum_{j \in [2]} \beta_{sj} }{ h+ c_s  } -2   \sum_{j \notin [s], [2]}^n  \frac{  \beta_{sj} }{(c_j - c_s )} ( \frac{1}{h+ c_s} - \frac{1}{h+ c_j}   ) \}. \nonumber \hspace{0cm}\end{eqnarray}

%$  \frac{X_i - X_s }{ (\zeta_s -\zeta_i) } =\\ \frac{(c_s +h ) ( c_i +h) }{ (c_i -c_s) }\{ \frac{ a_i^2}{(h+ c_i)^2}A_i   + \frac{ 2a_2 a_i }{ h+ c_i  } \sum_{j \in [2]}^n \beta_{ij}   +2a_i   \sum_{j \notin [i], [2]}^n   \frac{ a_j  \beta_{ij} }{(c_j - c_i )} ( \frac{1}{h+ c_i} - \frac{1}{h+ c_j}   ) \\   -  \frac{ a_s^2}{(h+ c_s)^2}A_s  - \frac{ 2a_2 a_s}{ h+ c_s  } \sum_{j \in [2]}^n \beta_{sj}  -2a_s   \sum_{j \notin [s], [2]}^n  \frac{ a_j  \beta_{sj} }{(c_j - c_s )} ( \frac{1}{h+ c_s} - \frac{1}{h+ c_j}   ) \}.$

In (\ref{aq0z3h1}),  $\frac{n_i d_i}{(h+ c_i)^2} $ is the only degree $-2$ term for the variable $h$ in
$ \zeta_i - \zeta_2$ in the sense of the paragraph just below Lemma \ref{L41}.
From (\ref{aq0z37}) and (\ref{aq01n5j})  we see  that there is no degree $-2$ term in $  \frac{X_i - X_s }{ (\zeta_s -\zeta_i) } $ as well as  $  \frac{X_2- X_s }{ (\zeta_s -\zeta_2) } $.
Then from  $0=Z_{2s}- Z_{is} = \zeta_i - \zeta_2 -  \frac{X_i - X_s }{ (\zeta_s -\zeta_i) } +   \frac{X_2 - X_s }{ (\zeta_s -\zeta_2) }$,  we get  $\frac{n_i d_i}{(h+ c_i)^2} =0$ by  Lemma \ref{L41}. We obtain $d_i=0$ for any $1< i \leq n$ with  $i \notin [2] $.
%Then, using (\ref{aq0z1}) and (\ref{aq0z3h1}), we get \begin{eqnarray} \label{aq0z3}\ \ \ \ \zeta_i =\frac{a}{2} -c c_i   +\frac{  cc_i^2 -ac_i + b  }{ { c_i} +h } +  \sum_{l \notin [2], [i] }  \frac{d_l}{c_i - c_l}\{ \frac{1}{(h+ c_l) }   - \frac{1}{ (h+ c_i) } \} \\+  \frac{3n_i d_i}{2(h+ c_i)^2} +  \sum_{l \notin [i], [2]}  \frac{d_l}{2(h+ c_l)^2}. \nonumber \hspace{3.4cm}\end{eqnarray}
Then (\ref{aq52}) becomes
 \begin{eqnarray} \label{aq4w}    h^{'}
=   ch^2 +ah + b,\end{eqnarray}
  and  from (\ref{aq0z1})$\sim$(\ref{aq0z3h1}),
 \begin{eqnarray} \label{aq0z3hxt2}
\zeta_2  =  -ch  -  \frac{a}{2  },  \hspace{4.8cm} \nonumber \\
\zeta_2^{'} +  \zeta_2^2  =\frac{a^2}{4} -  cb, \hspace{4.2cm} \\
\zeta_s =
\frac{a}{2} -c c_s   +\frac{  cc_s^2 -ac_s + b  }{  c_s +h },  \ \ {\rm for} \ 1< s \notin [2]. \nonumber
\end{eqnarray}
 We put (\ref{aq0z3hxt2}) and (\ref{aq01n5j})   into  (\ref{riis}) with $[i]=[2]$ and  $[s] \neq [2]$ and obtain
%$\zeta_i^{'} =-h^{'}\frac{  cc_i^2 -ac_i + b  }{ (c_i +h)^2 } =-( cc_i^2 -ac_i + b )\frac{ ch^2 +ah + b }{ (c_i +h)^2 }=-( cc_i^2 -ac_i + b )\{c +   \frac{a-2cc_i}{h + c_i }   + \frac{b- ac_i + c c_i^2 }{(h+ c_i)^2}    \} $.
%$\zeta_i^2 =(\frac{a}{2} -c c_i)^2   +\frac{  (cc_i^2 -ac_i + b )^2 }{(c_i +h)^2 }   +  (a -2c c_i )(\frac{  cc_i^2 -ac_i + b  }{ \hat{ c_i} +h })   $
%$\frac{a^2}{4} -  cb= \zeta_i^{'} +  \zeta_i^2 = -( cc_i^2 -ac_i + b )\{c +   \frac{a-2cc_i}{h + c_i }   + \frac{b- ac_i + c c_i^2 }{(h+ c_i)^2}    \} +(\frac{a}{2} -c c_i)^2   +\frac{  (cc_i^2 -ac_i + b )^2 }{(c_i +h)^2 }+  (a -2c c_i )(\frac{  cc_i^2 -ac_i + b  }{ \hat{ c_i} +h }) \\% = -c( cc_i^2 -ac_i + b ) +(\frac{a}{2} -c c_i)^2   $ trivial.
%$\zeta_2  =  -ch  -  \frac{a}{2  }= - \frac{h^{''} }{2h^{'}  } $  $  2chh^{'} + ah^{'}=  h^{''} $ so that  $ch^{2} + ah +c_0=  h^{'} $ for a constant $c_0$.
% So, we may write, for any $ i >1$, ?? $\zeta_i = \frac{a}{2}  + \sum_{l>1} \frac{e_{i, l}}{h+ c_l}$ by defining $e_{i, l} =0 $ when $i \in [2]$.
%$f^{'}=- (\zeta_2+\zeta_s) +   \sum_{1<j\leq n} \zeta_j     + \frac{X_2 - X_s }{ (\zeta_s -\zeta_2) },   \ \ \ \     {\rm for } \ \ [2] \neq [s].$
 \begin{eqnarray} \label{aq0z3hx}
f^{'}=( n_{2} -1)(-ch  -  \frac{a}{2  })  + ( n_{s} -1 ) ( \frac{a}{2} -c c_s   +\frac{  cc_s^2 -ac_s + b  }{ { c_s} +h }  )  \hspace{1cm} \\ +   \sum_{1<j \notin [2], [s]}  (\frac{a}{2} -c c_j  +\frac{  cc_j^2 -ac_j + b  }{ { c_j} +h } )    +(h + c_s )  (A_2- \frac{A_s  }{(h+ c_s)^2} + Y_2- Y_s). \nonumber
\end{eqnarray}
%In (\ref{aq0z3hx}), the expression  $\sum_{[j] }$ means that the summation is over equivalence classes, not elements in them.
%where $Y_2  =  \sum_{1< j \notin  [2]}^n   2a_2 a_j \frac{ \beta_{2j} }{ h+ c_j  }$, and for $ i \notin [2]$,$ Y_i =  \frac{ 2a_2 a_i }{ h+ c_i  } \sum_{j \in [2]} \beta_{ij}  +2a_i  \sum_{1< j \notin [i], [2]}   \frac{ a_j  \beta_{ij} }{(c_j - c_i )} ( \frac{1}{h+ c_i} - \frac{1}{h+ c_j}   )$.   Each of $Y_s, Y_t, Y_2  $ is of the form $ \sum_{l \notin [2]} \frac{a^l}{h+ c_l} $ for  constants $a^l$.

  Recalling $Y_2$ and $Y_s$ from (\ref{aq01n5j}), we have

$\{ (h + c_s )  ( Y_2- Y_s) \}^{'} = -\sum_{1< j \notin  [2], [s]}^n
   2 \frac{ \beta_{2j}(-c_j  + c_s ) }{( h+ c_j)^2  }  h^{'}
 +2  \sum_{1< j \notin [2], [s]}
   \frac{   \beta_{sj} }{  (h+ c_j)^2 }h^{'} $.

\smallskip
Using this, we get from (\ref{aq0z3hx}) and (\ref{aq4w})

$  f^{''} = \frac{dh}{ds} \frac{d f^{'} }{dh}= (ch^2 + ah +b) \{  -c ( n_{2} -1)  -( n_{s} -1 )  \frac{  cc_s^2 -ac_s + b  }{ (  c_s +h )^2}  -  \sum_{1<j \notin [2], [s]}   \frac{  cc_j^2 -ac_j + b  }{ ( c_j +h)^2 }     +A_2 + \frac{A_s  }{(h+ c_s)^2}-  \sum_{1< j \notin  [2], [s]}^n
   2 \frac{ \beta_{2j} (-c_j + c_s )}{ (h+ c_j )^2 }
  +2\sum_{1< j \notin [2], [s]}
        \frac{\beta_{sj }}{(h+ c_j)^2}   \} $.

\medskip
 Set  $D_s :=   -( n_{s} -1 )  (  cc_s^2 -ac_s + b)    + A_s   $,  and

    $ E_{j, s} :=    (  cc_j^2 -ac_j + b )+
   2 \beta_{2j} (-c_j + c_s )
-2 \beta_{sj} $. With these, we rewrite

$ f^{''} = (ch^2 + ah +b) \{  -c ( n_{2} -1)   +A_2  +  \frac{  D_s}{(h+ c_s)^2} -  \sum_{1<j \notin [2], [s]} \frac{  E_{j, s} }{ ( c_j +h)^2 }  \} $.

%Note that for $[l_0] \neq  [s]$,  $ \sum_{j \in [l_0]}   E_{j, s}$  is independent of $j$ in[l_o]$.

As  $ \frac{  ch^2 + ah +b  }{ (  c_s +h )^2} = c + \frac{a - 2c c_s}{h+ c_s }  +  \frac{c c_s ^2 - ac_s +b }{(h+ c_s )^2} $, we get
 \begin{eqnarray} \label{aq0z3hx8}
 f^{''} = (ch^2 + ah +b)  \{  -c ( n_{2} -1)   +A_2 \} \hspace{4.5cm} \\ +  D_s\{  c + \frac{a - 2c c_s}{h+ c_s }  +  \frac{c c_s ^2 - ac_s +b }{(h+ c_s )^2} \}
 -  \sum_{1<j \notin [2], [s]}   E_{j, s}  \{ c + \frac{a - 2c c_j}{h+ c_j }  +  \frac{c c_j ^2 - ac_j +b }{(h+ c_j )^2} \}. \nonumber
\end{eqnarray}
We also have $ f^{''} =\lambda + (n-1)(\frac{a^2}{4} -cb) $ which comes from (\ref{x1}) and (\ref{x2});
 $\zeta_2^{'} +  \zeta_2^2   =\frac{ f^{''} - \lambda}{n-1}$.
 So,
   \begin{eqnarray} \label{aq0z3hx82}
\lambda + (n-1)(\frac{a^2}{4} -cb) = (ch^2 + ah +b)  \{  -c ( n_{2} -1)   +A_2 \} \hspace{1.5cm} \\ +  D_s\{  c + \frac{a - 2c c_s}{h+ c_s }  +  \frac{c c_s ^2 - ac_s +b }{(h+ c_s )^2} \}
 -  \sum_{1<j \notin [2], [s]}   E_{j, s}  \{ c + \frac{a - 2c c_j}{h+ c_j }  +  \frac{c c_j ^2 - ac_j +b }{(h+ c_j )^2} \}. \nonumber
\end{eqnarray}

Lemma \ref{L41} and (\ref{aq0z3hx82}) first give $ c \{  -c ( n_{2} -1)   +A_2 \} = a \{  -c ( n_{2} -1)   +A_2 \}=0  $. If $ -c ( n_{2} -1)   +A_2 \neq  0$, then $a=c=0$ and $\zeta_2 =0$ from (\ref{aq0z3hxt2}), which is a contradiction to the assumption of this section. So, we get   $A_2 = c ( n_{2} -1).$

 Lemma \ref{L41} and (\ref{aq0z3hx82}) give, for  $1< s \notin [2]$,
 $(a-2cc_s)  D_s  =0 $ and
 $(c c_s ^2 - ac_s +b )  D_s  =0$.   If $D_{s} \neq0$, then $a-2cc_s=c c_{s} ^2 - ac_{s} +b =0$, so that   $\zeta_{s} =0$, a contradiction.  So, $D_{s} = 0$.
 Lemma \ref{L41} and (\ref{aq0z3hx82}) also give, for  $ [l_0] \neq [2], [s]$,
 if $ \sum_{j \in [l_0]} E_{j, s}  \neq 0$,  then $a - 2c c_j=c c_j ^2 - ac_j +b =0$ gives $\zeta_j =0$, a contradiction.
 So,  $ \sum_{j \in [l_0]} E_{j, s} = 0$ and (\ref{aq0z3hx82}) reduces to $\lambda + (n-1)(\frac{a^2}{4} -cb)=0$.
 As $f^{''} =\lambda + (n-1)(\frac{a^2}{4} -cb)=0$, so $f^{'} = f_0$, a nonzero constant.

%\begin{eqnarray}   -  \sum_{1<j \notin  [2], [s]} c E_{j, s}    =\lambda + (n-1)(\frac{a^2}{4} -cb), \hspace{2.1cm} \label{k03}\\  \sum_{j \in [l_0]} (a - 2c c_j)E_{j, s}  =0 \ \ {\rm for}  \ \ \ [l_0] \neq [2], [s],  \hspace{2cm} \label{k04} \\  \sum_{j \in [l_0]} (c c_j ^2 - ac_j +b )  E_{j, s}  =0 \ \ {\rm for}  \ \ \ [l_0] \neq [2], [s],\hspace{1.3cm} \label{k05}\end{eqnarray}so that if $ \sum_{j \in [l_0]} E_{j, s}  \neq 0$ for $ [l_0] \neq [2], [s]$,then $a - 2c c_j=c c_j ^2 - ac_j +b =0$ gives $\zeta_s =0$, a contradiction. So,  $ \sum_{j \in [l_0]} E_{j, s} = 0$ and (\ref{aq0z3hx82}) reduces to   $0=\lambda + (n-1)(\frac{a^2}{4} -cb)$.

We recall
the formula (1.34) in \cite{Chow}:  $R+ (f^{'} )^2 - 2 \lambda f =\hat{c}$ for a constant $\hat{c}$.  From (\ref{x1}) and (\ref{x1z}) we get $ R = n \lambda   -\sum_{i>1}\zeta_i{f^{'}}- f^{''}=  n \lambda   -\sum_{i>1}\zeta_i{f_0}$. So,
\begin{eqnarray}
- (f_0)^2 +2\lambda f+ \hat{c}  = n \lambda -\sum_{i>1}\zeta_i{f_0}. \label{k0569}
\end{eqnarray}

From $f^{'} = \frac{df}{dh}  \frac{dh}{ds} = f_0$ and (\ref{aq4w}),
we get $\frac{df}{dh}  = \frac{f_0}{ch^2 + ah +b}$.
We take derivative $\frac{d}{dh}$ on (\ref{k0569}) and  use (\ref{aq0z3hxt2}) to get,
 \begin{eqnarray} \label{62}
 \frac{2\lambda}{ch^2 + ah +b}  =  n_2 c + \sum_{1<j \notin [2]} \frac{  cc_j^2 -ac_j + b  }{ ({ c_j} +h )^2}.
\end{eqnarray}

We can prove
\begin{prop} \label{pro3}
Any  gradient Ricci soliton  with harmonic Weyl curvature has less than three distinct $\lambda_i$'s  among $\ i \geq 2$ in a refined adapted frame field.
\end{prop}
\begin{proof}
If there are at least three distinct $\lambda_i$'s  among $i \geq 2 $, then we have at least three distinct $c_i$'s. If $c=0$, by Lemma \ref{L41} applied to  (\ref{62}), $-ac_j + b =0$ for $[j] \neq [2]$, so that $a=b=0$ and $\zeta_2 =0$, a contradiction. So, $c \neq 0$.

Let $c_{j_1},  \cdots c_{j_m}$, $m \geq 2$,  be all the distinct $c_j$ values among  $1< j \notin [2]$.
If $-c_{j_1}$ is  a multiple root of $ch^2 + ah +b=0$,
then we multiply  $ (h+ c_{j_k})^2$, $k \geq 2$ on
 (\ref{62}) and then put $h=- c_{j_k}$ to get $cc_{j_k}^2 -ac_{j_k} + b =0$, a contradiction.  So, none of  $-c_{j_1}, \cdots , -c_{j_m} $ can be  a multiple root of $ch^2 + ah +b=0$.
Then multiply  $\Pi_{k=1}^m (h+ c_{j_k})^2$ on
 (\ref{62}) and then put $h=- c_{j_k}$ to get $cc_{j_k}^2 -ac_{j_k} + b =0$.   Now (\ref{62}) reduces to
$ \frac{2\lambda}{ch^2 + ah +b}  =  n_2 c$. We get $c=0$, a contradiction. The proposition is proved.
\end{proof}

\section{ Exactly two distinct $\lambda_i$'s  among $i \geq 2$ }
Let $(M^n,g,f)$, $n \geq 4$, be an $n$-dimensional gradient Ricci soliton with harmonic Weyl curvature and exactly  two distinct $\lambda_i$'s   among $i \geq 2$ in a refined adapted frame field. Then we cannot use
the argument of previous section which is based on (\ref{aq01n}). Here we start with
%Considering Lemma \ref{gg}  we only have to prove  the case when  $\zeta_i \neq 0$ for any $i >1$.

%We may assume $\lambda_2 = \cdots =\lambda_{k} \neq  \lambda_{k + 1}   = \cdots = \lambda_n $.

%$\nabla_{F_1 }   F_i=0 \ \ {\rm for}  \ \  i >1.$

%From$\Gamma_{ij}^k=0  \ \  {\rm for}   \ \  [i]= [j] \neq [k].$$g( \nabla_{F_i} F_i, F_1 )=  -\zeta_i $.

% Similarly, the span of $E_1, E_{k+1}, \cdots, E_{n}$ as well as  that  of $E_{2}, \cdots, E_{k}$ is integrable.

\begin{lemma}\label{77b1}
For  an $n$-dimensional gradient Ricci soliton  $(M^n,g,f)$, $n \geq 4$, with harmonic Weyl curvature,
suppose that
there is a refined adapted frame field $F_i$ on an open subset $V \subset M$ with
$\lambda_2 =\cdots =\lambda_{k} \neq \lambda_{k+1}=\cdots =\lambda_n$ for $2 \leq k \leq n-1$.
%, and  $\zeta_i \neq 0$ for any $i >1$.

Then
 there exist coordinates $(x_1:=s, x_2, \cdots,  x_n)$ in a neighborhood of each point in $V$
 such that $\nabla s= \frac{\nabla f }{ |\nabla f |}$ and $g$ can be written as
\begin{equation} \label{mtr1a}
g=ds^2+p(s)^2\tilde{g_1}+q(s)^2\tilde{g_2},
\end{equation}
 where  $p:=p(s)$ and $q:=q(s)$ are smooth functions and
 $\tilde{g}_i$, $i=1,2$, is a pull-back of an Einstein metric on a $(k-1)$-dimensional domain $N^{k-1}$ with $x_2, \cdots,  x_k$ coordinates, and on an $(n-k)$-dimensional domain  $U^{n-k}$  with  $x_{k+1}, \cdots,  x_n$ coordinates, respectively.

 We  have   $F_1 =\frac{\partial }{\partial s} $, $F_i =\frac{1}{p} e_i $, $i=2,\cdots, k$ and  $F_j =\frac{1}{ q} e_j $, $j=k+1, \cdots,n$  where $\{ e_i\} $  and  $\{ e_j\}$  are orthonormal frame fields of $\tilde{g}_1$ and  $\tilde{g}_2$, respectively.
\end{lemma}
\begin{proof}

This proof is similar to that of  Lemma 4.3  in \cite{Ki}, but is more simplified.
From (\ref{lambda06ax}), (\ref{y03}) and  (\ref{e07u}),
the span $D_1$ of $F_1, F_2, \cdots, F_{k}$ is integrable and $D_2$ of $F_{k+1}, \cdots, F_n$ is integrable.
Note that Lemma 4.2 of  \cite{Ki} can hold without dimensional restriction on $D_1$ and $D_2$.
So, the integrability of $D_1$ and $D_2$ here implies that  the metric $g$ can be written in some coordinate $y_i$ as follows;

$g= \sum_{i, j=1}^{k} g_{ij} dy_i \odot dy_j  +  \sum_{i, j=k+1}^{n} g_{ij} dy_i \odot dy_j  $, where $\odot$ is the symmetric tensor product and, for the coframe field  $\omega_i$ dual to $F_i$,  $\sum_{i=1}^{k}\omega_i^2       =\sum_{i, j=1}^{k} g_{ij} dy_i \odot dy_j  $ and  $\sum_{i=k+1}^{n}\omega_i^2= \sum_{i, j=k+1}^{n} g_{ij} dy_i \odot dy_j  $.
Set $ \partial_i:= \frac{ \partial }{\partial y_i} $ and $ g_{ij}= g( \partial_i , \partial_{j}  )$.
 For $i, j \geq k+1$ and $ 2 \leq l  \leq k $,  we get $\langle \nabla_{\partial_{i}}  \partial_{j} ,  \partial_l \rangle =0$ from (\ref{y03}). So,
 \begin{eqnarray} \label{bbaa}
0 & =     \langle \nabla_{\partial_{i}}  \partial_{j} ,  \partial_l\rangle =   \sum_{m=1}^n \langle\Gamma^{m}_{ij}  \partial_m , \partial_{l} \rangle \nonumber \hspace{2.3cm} \\
 &=   \sum_{m, a=1}^n \langle \frac{1}{2} g^{ma}( \partial_i g_{a{j}}+ \partial_j g_{a{i}} - \partial_a g_{ij} )\partial_m ,\partial_l \rangle \hspace{0.4cm}   \nonumber   \\ & =- \sum_{m, a=1}^n  \frac{1}{2} g^{ma}  \partial_a g_{ij} \langle \partial_m , \partial_l  \rangle   =  -\frac{1}{2}  \partial_{l}( g_{ij}). \hspace{0.6cm}
\end{eqnarray}
Similarly,  from the totally umbilicality of $D_2$ by Lemma  \ref{abc60x} {\rm (iv)} and $\zeta_n =-  \langle \nabla_{F_i}F_i, F_1 \rangle  $,
\begin{eqnarray*}
\zeta_n  g_{ij}  =  - \langle \nabla_{\partial_{i}}  \partial_{j} ,  \frac{\partial }{\partial s}\rangle = -   \sum_{m=1}^n \langle\Gamma^{m}_{ij}  \partial_m ,  \frac{\partial }{\partial s}  \rangle = \frac{1}{2} \frac{\partial }{\partial s} g_{i{j}}.
\end{eqnarray*}
 Integrating it, we get $ g_{ij} =  \pm e^{C_{ij}} q(s)^2$ for a function $q=q(s)$ if $g_{ij} \neq 0$. Here the function $q>0$ is independent of $i,j$ and each function $C_{ij}$ depends only on $y_{k+1}, \cdots, y_n$ by (\ref{bbaa}). This shows that $\sum_{i=k+1}^{n}\omega_i^2= \sum_{i, j=k+1}^{n} g_{ij} dy_i \odot dy_j  =q(s)^2\tilde{g_2}$ for some Riemannian metric $\tilde{g_2}$ on a domain with coordinates  $y_{k+1}, \cdots, y_n$.
  By similar argument we can show  $\sum_{i=2}^{k}\omega_i^2  =  p(s)^2\tilde{g_1}$ for a function $p(s)$ and a Riemannian metric $\tilde{g_1}$  on a domain with coordinates  $y_{2}, \cdots, y_k$.

As $F_1 = \nabla s$ by Lemma \ref{threesolbx},  $ds = g(F_1, \cdot )$. So, $\omega_1 = ds$ and we obtain  $g= ds^2 +   p(s)^2 \tilde{g_1} +   q(s)^2 \tilde{g_2} $.
Finally, $\tilde{g_i} $ can be proved to be Einstein mimicking the proof of  Lemma 5.1  in \cite{Ki}. The last clause of Lemma holds clearly.
\end{proof}

We set $a:=  \zeta_2 $ and $b :=  \zeta_n $. One can check that $a = \frac{p'}{p} $ and  $b  = \frac{q' }{q} $.  If we write the Ricci tensors of Einstein metrics $\tilde{g}_i$ in Lemma \ref{77b1} as $r^{\tilde{g}_1} =  (k-2) k_2 \tilde{g}_1$ and    $r^{\tilde{g}_2} =  (n-k-1) k_n \tilde{g}_2$, for numbers $k_2$ and $k_n$,
then the Ricci tensor components $R_{ij} = R(F_i, F_j)$ of $g$ can be directly computed as follows.
  \begin{eqnarray} \label{560}
    R_{11}=-(n-1)(a'+a^2), \hspace{6.5cm}\nonumber  \\
    R_{ii}=-a'-(k-1)a^2-(n-k)ab+  \frac{(k-2)}{p^2} k_2, \ \ \ {\rm for} \  2 \leq  i \leq k, \hspace{0.6cm}  \\
    \ \ \ \ \  \ \ R_{jj}=-b'-(n-k)b^2-(k-1)ab+  \frac{(n-k-1)}{q^2} k_n, \  \ {\rm for} \  k+1 \leq j \leq n. \nonumber
  \end{eqnarray}
Setting $P:= \frac{(k-2)}{p^2} k_2$ and $H:= \frac{(n-k-1)}{q^2} k_n$,   the formula (\ref{x1z}) gives
\begin{eqnarray}
  -a^{'}-(k-1)a^2-(n-k)ab+  P=-a f^{'} + \lambda,   \label{fromwpe} \\
 -b^{'}-(n-k)b^2-(k-1)ab+  H=-b f^{'} + \lambda, \ \label{fromwpebb}
\end{eqnarray}

%We note that  (\ref{fromwpe}) and (\ref{fromwpebb}) can be obtained also  from (\ref{rii}) and  (\ref{aq01})$\sim$(\ref{aq03b}),

%We note that  (\ref{fromwpe}) and (\ref{fromwpebb}) can be obtained from(\ref{rii}) and  (\ref{aq01}),   (\ref{aq02}) and  (\ref{aq03}) (\ref{aq01b}),   (\ref{aq02b}) and  (\ref{aq03b})

% \begin{eqnarray} \label{aq77}\lambda - b f^{'}=R_{nn}   = -b^{'} -b^2  -b  ( -b+ n_2a + n_n b  ) + \frac{h^{'} }{(h+ c_n)^2}A_n,\end{eqnarray}

%where for $i=2, n$,   $A_i :=  \sum_{j \in [i], j \neq i} \{   a_i e_i {\gamma}_{jj}^i  - a_je_j {\gamma}_{ii}^j - ({\gamma}_{ii}^j)^2- ({\gamma}_{jj}^i)^2   - \sum_{k \neq 1, i, j , k \in [i]} {\gamma}_{jj}^k {\gamma}_{ii}^k+ \sum_{k \neq 1, i, j, k \in [i]}  ( -{\gamma}_{ij}^k  {\gamma}_{jk}^i -{\gamma}_{ij}^k{\gamma}_{kj}^i +{\gamma}_{ji}^k{\gamma}_{kj}^i )  \}.  $

We see that $P^{'} = -2 a P$ and $H^{'} = -2 b H$.
From  (\ref{rii}) and   (\ref{560}),    we get $X_2 = P$ and $X_n = H$.  Then  (\ref{riis}) gives
 \begin{eqnarray} \label{riis2}
f^{'}   = ( k-2)a  + ( n-k -1 ) b    +\frac{H-P }{ (a -b) }.
\end{eqnarray}
Put it into (\ref{fromwpe}) and get

%$\lambda -  \zeta_2 [ ( n_{[2]} -1)\zeta_2  + ( n_{[i]} -1 ) \zeta_i    - \{ \frac{ a_i^2}{(h+ c_i)}A_i - ({ c_i} +h)a_2^2 A_2\}  ]   = -\frac{u_2^{''} }{u_2 }  - \{    ( n_{[2]} -1)\zeta_2^2  + n_{[i]} \zeta_2 \zeta_i \}+   h^{'} a_2^2 A_2.$

%$\lambda - \frac{h^{''} }{2(h+c_i)  }       - \frac{h^{''} }{2 h^{'} }\{ \frac{ a_i^2}{(h+ c_i)}A_i - ({ c_i} +h)a_2^2 A_2\}  = -\frac{h^{'''} }{2h^{'} }+ \frac{(h^{''})^2}{2(h^{'})^2}   +   h^{'} a_2^2 A_2$; an ODE hard to solve explicitly.  So, we shall do alternatively.

\begin{eqnarray} \label{k4}
\lambda +a^{'}+ a^2+ ab  + \frac{ a  }{  b -  a } H -\frac{ b  }{   b -  a }P
   =0.
\end{eqnarray}
Take derivative $\frac{d}{ds}$ on (\ref{riis2}),  and use $f^{''}  = \lambda - \lambda_1= \lambda + (n-1) (a^{'} + a^2 )$ and $b^{'} = a^{'} + a^2 -b^2$ to get
 \begin{eqnarray} \label{riis24t}
\lambda +2a^{'} +ka^2 +( n-k -1 ) b^2   -P-H =0.
\end{eqnarray}
To remove $P$ and $H$, we compute $(a+ b) \times$(\ref{riis2})$-2 \times$(\ref{k4})$+$(\ref{riis24t}) to obtain;
 \begin{eqnarray} \label{riis24i}
(a+b) f^{'} =( n-1) ab +\lambda.
\end{eqnarray}

 Take derivative $\frac{d}{ds}$ on (\ref{riis24i})  and use $f^{''} = \lambda + (n-1) (a^{'}+a^2  )$ to get

$ (2a^{'}+a^2 - b^2) f^{'} + (a+b) (\lambda + (n-1) (a^{'}+a^2  ))  =( n-1)\{ a^{'}b +(a^{'}+ a^2 - b^2 ) a\}.$

\smallskip
\noindent Multiply the above by $(a+b)$ and use  (\ref{riis24i}) to get
 \begin{eqnarray} \label{riis24ip}
(a^{'}+a^2 + ab  ) \{ ( n-1) ab +\lambda \}=0.
\end{eqnarray}

\medskip
\begin{lemma} \label{apb}
It holds that $a+b \neq 0$.
\end{lemma}

\begin{proof}
If  $a+b =0$, then from  $a^{'} +  a^2   =  b^{'}   + b^2$, we get $a^{'}    =  b^{'}=0$. If $ab=0$ further, then $a=b=0$, a contradiction to $ \lambda_2 \neq \lambda_n$. So, $ab \neq 0$.
 Then (\ref{k4}) gives $P+H = 2 \lambda$.  And $0= P^{'}+H^{'}  =  -2aP -2bH = -2a(P-H)$ so that $P=H=\lambda$.  By (\ref{riis2}), $f'$ is a constant. Now, $ 0= f^{''} = -R_{11}+ \lambda= (n-1)a^2+ \lambda$.
 By  (\ref{riis24i}),
  $ -(n-1)a^2 +\lambda= 0.$ So, $a=0$, a contradiction. This proves the lemma.
\end{proof}

If  $( n-1) ab +\lambda=0$ in  (\ref{riis24ip}), then $a+b =0$ by  (\ref{riis24i}), which cannot occur.
 We get
%From $a^{'} + a^2 = b^{'} + b^2$, $a^{'} = b^{'} =0$. If $a=0$ then $b=a=0$, a contradiction. So, $a, b \neq 0$. (\ref{k4}) gives $H+P =2 \lambda$. And $0= P^{'}+H^{'}  =  -2aP -2bH = -2a(P-H)$ so that $P=H=\lambda$.  By (\ref{riis2}), $f'$ is a constant. Now, $ 0= f^{''} = -R_{11}+ \lambda= (n-1)a^2+ \lambda$.As $0= ( n-1) ab +\lambda=-( n-1) a^2 +\lambda$,we get $a=0$, a contradiction.
 \begin{eqnarray} \label{riis24ip2}
a^{'}+a^2 + ab =0.
\end{eqnarray}

\bigskip

If $a=0$ or $b=0$, we shall use Lemma \ref{gg}.
We now assume  that $a \neq 0$ and $b \neq 0$.
As $a = \frac{p'}{p} $ and  $b  = \frac{q' }{q} $,   (\ref{riis24ip2}) gives
  $qp^{''} + q^{'} p^{'} =0$. Then,
 \begin{eqnarray}
p^{'}= \frac{c_1}{q}, \ \ {\rm for \  a \ constant} \ c_1 \neq 0. \label{fromwpe36}
\end{eqnarray}
 By $a^{'}  +  a^2 = b^{'}  +  b^2  $, we also get

  \begin{eqnarray}
q'= \frac{c_2}{p}, \ \ {\rm for \  a \ constant} \ c_2 \neq 0. \label{fromwpe37}
\end{eqnarray}
 Then  $ \frac{c_2}{c_1}\frac{p'}{p} = \frac{q'}{q} $, i.e.  $ c_3 a = b$, $c_3 \neq 0$.
 Note that $c_3 \neq -1$ by Lemma \ref{apb}.

% From (\ref{fromwpe31}), $a =c$, a nonzero constant. And $b =-c$. Then $p= c_p e^{cs} $ and  $h= c_h e^{-cs}$ for positive constants $c_p$ and $c_h$. (\ref{fromwpe27}) gives$  \lambda  = (n-1)c^2 $.And $ f^{''} = 2(n-1) c^2 $.   $f'  = 2(n-1) c^2 s +c_6 $.Then (\ref{fromwpe}) gives$  -(k-1)c^2+(n-k)c^2+  \frac{(k-2)}{c_p^2 e^{2cs}} k_2=-c (  2(n-1) c^2 s +c_6  ) + (n-1)c^2,$ from which we get$c=0$ and $p$ is a constant, a contradiction to $\zeta \neq 0$.

From (\ref{riis24ip2}) we get $a^{'} = -(1+c_3)a^2   $.
So,  $ a = \frac{p^{'}}{p}= \frac{1}{(1+c_3) s}$,  after a translation of $s$ by a constant.
 We have $p= c_5  s^{\frac{1}{1+c _3}} $ and
$p'= \frac{c_5}{1+ c_3} s^{\frac{1}{1+c _3} -1 }   $ for a constant $c_5 \neq 0$.
From (\ref{riis24i}), $ a f^{'}=\frac{(n-1) c_3 a^2+     \lambda}{  (1+ c_3 )}$. We put all these into (\ref{fromwpe}) and  get

$
 Q(c_3) (p^{'} )^2+  (k-2) k_2(1+ c_3 )=c_3     \lambda p^2, $ where $ Q(c_3):= (1-n+k) c_3^2 + 2 c_3 + 2-k.$
 Differentiating this and dividing by $2pp^{'}$  gives
$  Q(c_3)\frac{ p^{''}}{p} =c_3 \lambda  $ and
$  Q(c_3)   =  - \lambda  (1+c_3)^2 s ^{2}$. Comparing both sides of the latter equality, $\lambda=0$ and $ Q(c_3) =0$. Viewing $ Q(c_3) =0$ as a quadratic equation of $c_3$, its discriminant $D$ equals
$D= 4\{1 - (n-k-1) (k-2)\}$.  As $D$ must be nonnegative and  $k$ is an integer with $2 \leq k \leq n-1$, so either
 $k=3$ and $n-k=2$,  or    $ (n-k-1)(k-2)=0$.

\smallskip
If $k=3$ and $n-k=2$, then $ Q(c_3)=0$ becomes $ - c_3^2 + 2 c_3 -1 =0$, so $c_3 =1$.  But then  $ a = b$, a contradiction to  $ \lambda_2 \neq \lambda_n$.
% (\ref{j64}) gives $k_2=0$. We can get $k_n =0$.

Therefore  $ (n-k-1)(k-2)=0$, i.e. $n-k-1=0$ or  $k=2$; these two cases are similar, so we may treat the $k=2$ case only. Then $Q(c_3)=0$ gives $  c_3 = \frac{2}{n-3}$.
If $n=5$, then $c_3=1$ and $a=b$, a contradiction. So, we need $n \neq  5$.   Now $p= c_5  s^{\frac{n-3}{n-1}} $  and
$q=c_6 s^{\frac{2}{n-1}} $ for a constant $c_6 \neq 0$ from (\ref{fromwpe36}). As $k=2$, we get $g= ds^2   + s^{\frac{2({n-3})}{{n-1}}} dx_2^2+  s^{\frac{4}{{n-1}}} \tilde{g}_2$ after absorbing the constants $c_5$ and $c_6$ into $dx_2^2$ and  $\tilde{g}$, respectively.
As $\tilde{g}_1$ is $1$-dimensional, $k_2=P=0$. Put (\ref{riis24ip2}) and $\lambda=0$ into   (\ref{k4}) to get
$ H  =  k_n=0$. So, $\tilde{g}_2$ is a Ricci flat metric.
 (\ref{riis2}) gives  $ f^{'}=\frac{2(n-3)}{(n-1) s}$,   so that $f = \frac{2(n-3)}{{n-1}}\ln|s| $ modulo a constant.

 %Conversely, one can readily check that the space $(g,f)$ as described in the above paragraph is indeed  a  gradient Ricci soliton with harmonic Weyl curvature.

% $ -a^{'}-(k-1)a^2-(n-k)ab+  \frac{(k-2)}{p^2} k_2=-  (k-2)a^2-(n-k-1) ab- \frac{ a}{b -a}( \frac{(k-2)}{p^2} k_2 -  \frac{(n-k-1)}{h^2} k_n ) + \lambda$.

\medskip
We are going to summarize the above discussion of this section.

\begin{prop} \label{prop3}
Let $(M^n,g,f)$, $n \geq 4$, be an $n$-dimensional (not necessarily complete) gradient Ricci soliton with harmonic Weyl curvature and exactly  two distinct $\lambda_i$'s  among $i \geq 2$ at each point of  $ M$.

Then, in a neighborhood $V$ of each point in the open dense subset $  M_r \cap \{ \nabla f \neq 0  \}$,
there are  some coordinates $(x_1 := s, x_2, \cdots, x_n)$ with $ \nabla s = \frac{\nabla  f}{|\nabla f|}$, in which
one of the following two cases holds;

{\rm (i)} $(V, g)$ isometric to
the Riemannian product $ (N_1^{k}, g_1 ) \times (N_2^{n-k}, g_2)$ of a $k$-dimensional Riewmannian manifold  $(N_1^{k}, g_1 ) $ with $x_1, \cdots,  x_k$ coordinates, and  an $(n-k)$-dimensional Riewmannian manifold  $(N_2^{n-k}, g_2)$  with  $x_{k+1}, \cdots,  x_n$ coordinates,.
Here $g_1$ is Ricci flat and $r_{g_2} =  \lambda g_2 $, $2 \leq k \leq n-2$. Moreover $f$  is the pull-back of a function  on $N_1^k$, denoted by $f$ again, and    $ (N_1^{k}, g_1, f )$  is itself a gradient Ricci soliton and
 $f = \frac{\lambda s^2}{2}$ modulo a constant.

%$g =ds^2+s^2 {g_1}+{g_2}$, where $ds^2+s^2 {g_1}$ is a Ricci flat Riemannian metric on a %$k$-dimensional domain, and $g_2$ is an Einstein metric with $r_{g_2} =  \lambda g_2 $, $\lambda %\neq 0$,  on a $(n-k)$-dimensional domain, $2 \leq k \leq n-2$. Moreover, $f = \frac{\lambda %s^2}{2}$ modulo a constant.

{\rm  (ii)}  We need  $n \neq 5$ and
\begin{equation} \label{metricg3}
g|_{V}= ds^2   + s^{\frac{2({n-3})}{{n-1}}} dx_2^2+  s^{\frac{4}{{n-1}}} \tilde{g},
 \end{equation}
 where ${\tilde{g}}$ is a Ricci flat metric on an $(n-2)$-dimensional manifold $N$ with coordinates $x_3, \cdots, x_n$.
 Moreover, $\lambda=0$ and $f=\frac{2(n-3)}{{n-1}} \ln |s| $ modulo a constant.
\end{prop}

\begin{proof}
%Recalling the definition of $M_r$ in Section 2, we have $M_r \cap \{ \nabla f \neq 0  \} ???=   \{ \nabla f \neq 0  \}$, which is dense in $M$ because $f$ is not constant and $(M^n,g,f)$ is real analytic in some coordinates.

Each point in the open dense subset $M_r \cap \{ \nabla f \neq 0  \} =  \{ \nabla f \neq 0  \}$ has a neighborhood
on which there is a refined adapted  frame field $F_i$ by Proposition \ref{pop1}. We may then use Lemma \ref{77b1} and the arguments afterwards;
after we have got (\ref{riis24ip2}), we have treated the case when $a \neq 0$ and $b \neq 0$ and  obtained (ii).
In the other case when $a=0$ or $b=0$, we can just use  Lemma \ref{gg} to get (i) ; here we have  $2 \leq k \leq n-2$ because there will be exactly  one  distinct $\lambda_i$  among $i >1$ if  $k=1$.

%???Remark  ?? Note that  when $n=5$, the metric in (\ref{metricg3}) has $\lambda_2 %\lambda_n$,sothese dimensions are excluded.
\end{proof}

%\begin{lemma} \label{rel}Let $c_l$, $l=1, \cdots , k$ be pairwise different real numbers.If $R(x) = \sum_{l=1}^k  p_l(x) \ln|x+ c_l|$ for $x$ in an open interval $I$ with  rational functions $R$ and  $p_l$, then $R(x)=0$ and $p_l(x) =0$. \end{lemma}

\noindent {\bf  Proof of Theorem \ref{locals}:}  If $f$ is a constant, then $(M, g)$ is an Einstein manifold. If $f$ is not a constant,
the subset $M_r \cap \{ \nabla f \neq 0  \} $ is open and  dense in $M$. Near each point $p_0$ in  $M_r \cap \{ \nabla f \neq 0  \} $, there is an adapted frame field $E_i$. By Proposition \ref{pop1} there exists a refined  frame field $F_i$ on a neighborhood $U$ of $p_0$.
By Proposition \ref{pro3},  $(U, g, f)$ has less than three distinct $\lambda_i$'s among $i \geq 2$ at each point.
By  Proposition \ref{prop3}, when $(U, g, f)$ has
exactly  two distinct $\lambda_i$'s  among $i \geq 2$ at each point, we obtain (ii) and (iv) of Theorem \ref{locals}.
%% for (iv), we have $f=\frac{2(n-3)}{{n-1}} \ln |s| $ in Proposition \ref{prop3} (ii), but when $s<0$, we may set $x_1 = - s$ to get the statement of (iv).

If $(U, g, f)$ has
exactly  one distinct $\lambda_i$  among $i \geq 2$ at each point, then either $\lambda_1 = \lambda_2 = \cdots  = \lambda_n$, or  $\lambda_1 \neq  \lambda_2 = \cdots  = \lambda_n$.  In the former case, it is Einstein and from the section 1 of \cite{CC}, $g$ can be written as
$g=  ds^2 +   ( h(s))^2 \tilde{g}$, where $h(s)=  f^{'}(s)$ and  $\tilde{g}$ is Einstein. So, this case belongs to (iii).
 In the latter case, we can easily get (iii) from the argument  in  the proof of  \cite[Proposition 7.1]{Ki}. This finishes the proof of Theorem \ref{locals}.

%In proving Theorem  \ref{steady}, we cannot use the afore-mentioned result of  \cite{CC1, CM} on complete locally conformally flat {\it steady} GRS because the Riemannian metrics of  (iii) of Theorem \ref{locals} are not always locally conformally flat. Moreover, we cannot mimick the proof of  \cite{CC1, CM} in our harmonic Weyl curvature case, as they essentially use the non-negative Ricci-curvedness of complete locally conformally flat {\it steady} GRS proved in \cite{Z}, but we do not have  such non-negativity. Instead, we shall use a comparison method  as they  have explicit warped product forms.

\section{Steady and expanding gradient Ricci solitons with harmonic Weyl curvature}

In this section we prove Theorem  \ref{steady} and Theorem \ref{expand}.  Theorem \ref{locals} shows all the possible four types (i)$\sim$(iv) of regions in
 a  gradient Ricci soliton $(M, g, f)$  with harmonic Weyl curvature.
If $(M, g, f)$ has an open region of type (i), then by real analyticity $f$ is constant on $M$,  and
$M$ itself is a type-(i) region.

For the next lemmas, near a point $p $ with $\nabla f (p)  \neq 0$, we recall the curve $\theta_p (t) := \theta(p, t) $  defined by $\theta_p(0) =p$ and $\theta_p^{\prime} (t) = \frac{\nabla f}{|\nabla f|} $.  By   Lemma \ref{threesolbx} (vi), $\theta_p (t) $ is a geodesic and can be extended for $t \in \mathbb{R}$ when $g$ is complete.  As a map, $\theta_p: \mathbb{R} \rightarrow M$ is real analytic from the geodesic equation.

\begin{lemma} \label{19}
A  complete gradient Ricci soliton $(M, g, f)$ with harmonic Weyl curvature cannot contain a type-{\rm (iv)} region of  Theorem \ref{locals}.
\end{lemma}

\begin{proof}  Suppose that   $(M, g, f)$  contains an open region $V$ of type {\rm (iv)}. Near $p \in V$,
we can write
  $g= ds^2   + s^{\frac{2({n-3})}{{n-1}}} dx_2^2+  s^{\frac{4}{{n-1}}} \tilde{g}
$
and  $f=\frac{2(n-3)}{{n-1}} \ln(| s|) $ modulo a constant. From the proof of  Theorem \ref{locals}, $s$ is a function with $\nabla s =  \frac{\nabla f}{ |\nabla f|}$.
  Consider the  function $h(t) := f  (\theta_p(t))$ for $t \in \mathbb{R}$.
 As  $\frac{d}{dt}  \theta_p (t)= \nabla s ( \theta_p (t) ) $, $s( \theta_p(t) ) =t +c $ for   a constant $c$ when $t \in I$ for an open interval $I$.
So, $h(t)=\frac{2(n-3)}{{n-1}} \ln|t+c| $ modulo a constant on $I$.
By the  real analyticity of $h$ on $\mathbb{R}$,
this is a  contradiction. We proved the lemma.
\end{proof}

 \begin{lemma} \label{lcf1v}
For a connected complete  gradient Ricci soliton $(M, g,f)$ with harmonic Weyl curvature, suppose that  $M$ contains  a type-{\rm (iii)} region. Then the followings   hold.

{\rm (i)} For any point $p \in M$ with $\nabla f (p) \neq 0$,  there are numbers $\tilde{a}<0<\tilde{b} $
such that    $\nabla f  \neq 0$  on $U:=\theta (\Sigma_c^0 \times (\tilde{a}, \tilde{b}) ):=\{ \theta_{q}(t)  \ |  \ q  \in \Sigma_c^0, \    t \in  (\tilde{a}, \tilde{b})  \}$ where $\Sigma_c^0$ is  the connected component of $f^{-1}(c)$ containing $p$.

{\rm (ii)} The region $U$ in {\rm (i)} is the connected component of $f^{-1}((a, b))$ which contains $p$, where  $a=f( \theta_p (\tilde{a} ))  $ and $b=  f( \theta_p (\tilde{b} )) $.

{\rm (iii)} On $U$ we have
\begin{equation} \label{metr}
g= ds^2 +    h(s)^2 \tilde{g}   \ \ \   \  \  s \in I=(\tilde{a},  \tilde{b}),
\end{equation}
where $h$ is a smooth positive function on $ I$ and $\tilde{g}$ is  an Einstein metric on $\Sigma_c^0$.

\end{lemma}

\begin{proof}
The statements (i) and (ii) hold readily from Lemma \ref{threesolbx} (ii) and the completeness of the Reimannian metric $g$.

As $M$ contains  a type-{\rm (iii)} region, $M$  cannot admit a region of  type-(i) or type-(iv).
It cannot have any type-(ii) region either, because of the continuity and the multiplicity  of Ricci eigenvalues.
So,  $M$ admits only type-(iii) regions whose union forms an open dense subset of $M$.

Suppose $ U \subset  M_{r}  $.  From the proof of Theorem  \ref{locals},  a neighborhood $B_p$ of each point $p$ in  $U \subset M_r \cap \{ \nabla f \neq 0  \} $ is a type-(iii) region where we can write  $g= ds^2 +    h(s)^2 \tilde{g} $ on $B_p$. As $s$ is well defined on $U$,  these local expressions of  $g$ match and we have
 (\ref{metr}) on $U$.

 When $ U  \not\subset  M_{r}  $, as $M_r \cap \{ \nabla f \neq 0  \} $ is dense in $M$, (\ref{metr}) still holds  on $U$ by  the smoothness and continuity. This yields (iii).
 \end{proof}

 Suppose $\tilde{b} < \infty$ in Lemma \ref{lcf1v} (iii). As $g$ is complete, when $s \rightarrow \tilde{b}-$,
 $h(s)$ in (\ref{metr})  is  bounded and the Ricci tensor component $ r(\nabla s, \nabla s) = - (n-1)\frac{{h}^{''}}{{h}}$
 is bounded so that   $h^{''}(s) $ is  bounded.
So,  $h(s)$ cannot fluctuate and
 $\lim_{s \rightarrow  \tilde{b}-} h(s) =h_0$ for some number $h_0 \geq 0$.

\begin{lemma} \label{geod55} Let $(M, g, f)$, $U $, $\Sigma_c^0$, $ (\tilde{a}, \tilde{b})$, $a$ and $b$  be  as in Lemma  \ref{lcf1v}  so that
$g= ds^2 +    h(s)^2 \tilde{g}$ on $U=\theta (\Sigma_c^0 \times (\tilde{a}, \tilde{b}) )$ where $\nabla f \neq 0$.
Suppose that $\tilde{b}< \infty $  and  $\nabla f =0$ on the set $W:=f^{-1}(b) \cap   \overline{U}$, where  $\overline{U}$ is the closure of $U$.
Then the following {\rm (i)} and  {\rm (ii)}   hold.

\medskip

 {\rm (i)}   If $h_0 := \lim_{s \rightarrow  \tilde{b}-} h(s)=0$, $\tilde{g}$ is a spherical metric and $W$ is a point, say $\{p_0\}$, so that $g$ is a smooth rotationally symmetric metric on  $U \cup \{ p_0 \}$.

{\rm (ii)} If $h_0 > 0$, then $W$ is an $(n-1)$-dimensional real analytic submanifold of $M$ and the geodesic $\theta_p (t)= \theta (p, t) $ with $p \in \Sigma_c^0$ meets $W$ perpendicularly at $t= \tilde{b}$.  Moreover, $W= \theta (\Sigma_c^0 \times  \{ \tilde{b} \} ) $ and
the  map $\theta : \Sigma_c^0 \times (- \infty, \infty) \longrightarrow M$  is real analytic  with a local real analytic inverse near each point of $\Sigma_c^0 \times  \{ \tilde{b} \}  $.

\end{lemma}

\begin{proof}
(i) can be seen with standard argument.
To prove (ii), suppose $ h_0 >0$. We  can view $\theta $ as the normal exponential map of $\Sigma_c^0$, i.e.    $\theta : \Sigma_c^0 \times (- \infty, \infty) \longrightarrow M$.
 As $f$ is real analytic,  $\Sigma_c^0$ is a real analytic submanifold, and so the map   $\theta : \Sigma_c^0 \times (- \infty, \infty) \longrightarrow M$ is real analytic.

 Clearly $ \theta (\Sigma_c^0 \times  \{ \tilde{b} \} ) \subset W$.  For $q_1 \in W$, there exists a sequence of points $p_n $ in $U$ converging to $q_1$.
The distance $ d(q_1, \Sigma_c^0   )$ from $q_1$  to $\Sigma_c^0 $ is realized by a minimizing geodesic segment $\gamma$  with length $\tilde{d_0}$, which meets  $\Sigma_c^0$ perpendicularly at a point $\tilde{p}_0 \in \Sigma_c^0 $.  Due to (\ref{metr}),    $d(p_n, \Sigma_c^0   )  \leq  \tilde{b}$. So,  $\tilde{d_0} = d(q_1, \Sigma_c^0   )  \leq  \tilde{b} $.
We have  $q_1 = \theta (\tilde{p}_0, \tilde{d_0})
$.  Clearly  $ \tilde{d_0} \geq \tilde{b}$ so that  $ \tilde{d_0} = \tilde{b}$.
We get $q_1 =\theta (\tilde{p}_0, \tilde{b} )$. So,  $W = \theta ( \Sigma_c^0 \times \{ \tilde{b} \}    )$.

%Using $h_0 >0$,

% For $p \in \Sigma_c^0$, $(p, u)  \in  \Sigma_c^0 \times (\tilde{a},  \tilde{b})$
We shall compute the derivative of $\theta$ at $(p, u)  \in  \Sigma_c^0 \times (\tilde{a},  \tilde{b}]$, i.e.
%$(p, u)  $  for $\tilde{a} <  u \leq  \tilde{b}$  ;
 $d \theta_{(p. u)} : T_{ (p, u)  }( \Sigma_c^0 \times (\tilde{a}, \infty) )  \rightarrow  T_{   \theta{(p, u)}}  M $.
We define  curves $c_i(t)$, $i =1, \cdots, n$,  in  $ \Sigma_c^0 \times (\tilde{a}, \infty)$ so that  $c_i(0) =(p, u)$ and
$v_i :=c_i^{'} (0)$  forms a basis of $ T_{ (p, u)  }(\Sigma_c^0  \times  (\tilde{a}, \infty) )   $ as follows; set
 $c_1(t) =  (p, u+ t)  $  and  let  $c_i(t)$, $i\geq 2$, lie in $\Sigma_c^0 \times \{ u \}$.
  As  $g= ds^2 +    h(s)^2 \tilde{g}$ on $U$, writing $w_{i}^u:= d \theta_{(p, u)} ( v_i)$, from the first paragraph of Section 5 we get  $g( w_{i}^u,  w_{j}^u  )  = h(u)^2 \tilde{g}(v_i , v_j )  $ for $\tilde{a} <  u < \tilde{b}$ where $2 \leq i, j \leq n$ and $g( w_{1}^u,  w_{i}^u )  = 0 $.
 When $u \rightarrow \tilde{b}-$, as $\theta$ is $C^{\infty}$,  $w_{i}^u  $ converges to $w_{i}^{\tilde{b}}  $  and  $g( w_{i}^{\tilde{b}},  w_{j}^{\tilde{b}} )  = h_0^2  \tilde{g}(v_i , v_j )  $ and $g(w_{1}^{\tilde{b}}, w_{i}^{\tilde{b}}  )  =  \delta_{1i}  $. We clearly have $g( w_{1}^{\tilde{b}}, w_{1}^{\tilde{b}} )  = 1  $.
Then for $\tilde{a} < u \leq  \tilde{b}$,   $w_i^u =d \theta_{(p, u)} (v_i) $   form a  basis of $ T_{\theta{(p, u)}  } M $, so
 $d \theta_{(p,u)}$ is a linear isomorphism  for each point of $\Sigma_c^0 \times (\tilde{a}, \tilde{b}] $.
Applying the inverse function theorem in the real analytic category \cite{KP},
 the map  ${ \theta}: \Sigma_c^0 \times (\tilde{a}, \tilde{b}] \longrightarrow M$ is  real analytic  with  local real analytic inverse near each point ${\theta} (p, t)$ of $U  \cup W$. Now $W$ is an $(n-1)$-dimensional real analytic submanifold of $M$.
The Riemannian metric $g$ on $U$ can be extended and  written as $ ds^2 +    h(s)^2 \tilde{g}$  on $U \cup W.$
 As $\theta_p^{'} (t) $, $\tilde{a} < t < \tilde{b}$,  is perpendicular to  level surfaces of $f$, so  $\theta_p^{'} (\tilde{b}) $  is perpendicular to $W$ which is the $C^{\infty}$ limit of  level surfaces of $f$.  This proves (ii).
\end{proof}

We have stated and proved Lemma \ref{geod55}  only for $\tilde{b}$, but similar statements clearly  hold for $\tilde{a}$.
We note that, even if $h_0 >0$,  $W$ may or may not be a boundary component of $U \cup W$,   since $\theta$ restricted to $ \Sigma_c^0 \times \{ \tilde{b}\}$ may or may not be one-to-one.
 The next example serves as a guidance.

\medskip

{\bf Example 1}
For an $(n-1)$-dimensional Riemannian manifold  $(\Sigma, g_0)$, suppose that there is  a fixed-point free  $g_0$-isometry $\phi$ with $\phi^2 =$ Identity map.
Let
$\Phi :  \Sigma \times \mathbb{R}  \longrightarrow   \Sigma \times \mathbb{R}  $ be the  map
defined by $\phi (x, s)  = ( \phi(x), -s)$.
$\Phi$ is an isometry  with respect to the product Riemannian metric $G = g_0 + ds^2$.  Let $M$ be the $\mathbb{Z}_2$-quotient Riemannian manifold $\Sigma \times \mathbb{R}  /  \Phi $ and $\pi$ be the projection $\pi : \Sigma \times \mathbb{R}  \longrightarrow M$.  Set the function  $f(x, s) = -s^2$ well defined on $M$. Set $U_0 :=\Sigma \times (-\infty,  0) $ and $V_0 :=\Sigma \times (0,  \infty) $ and $U= \pi( U_0)$,  $V= \pi( V_0)$ and $W= \Sigma \times \{0\}/ \Phi  $. In $M= U \cup W \cup V$, we have  $U=V$.   We consider the local one-parameter group action $\theta(p, t)$ associated  to  $ \frac{  \nabla f }{|  \nabla f |}$ where $p  \in ( \Sigma \times \{ -1\}) /  \Phi \subset U$.  $\theta_p(t)$ is a geodesic.
 On $U_0$ representing $U$, $ \frac{  \nabla f }{|  \nabla f |} = \nabla s$, and writing $p$ by $(q, -1)$ with $q \in \Sigma$, we get  $\theta_p(t) = (q, -1 +t)$ for $t <1$.  The geodesic curve $\theta_p(t)$ meets $\Sigma \times \{0\}$ perpendicularly at $(q, 0)$.
 On $V_0$ representing $V$, $ \frac{  \nabla f }{|  \nabla f |} =- \nabla s$.
The  submanifold $W$ is not a boundary submanifold of $U \cup W=M$, but lies in the interior of $M$.

\bigskip
  We recall the {\bf  Weierstrass preparation theorem}  \cite{KP}: Let  $f(x_1 , \cdots ,x_k) $ be a  real analytic function in a neighborhood of the origin in $\mathbb{R}^k$ and assume that  $f(0, \cdots, 0, x_k )\neq 0.$ Then $f$ may be written in the form $f= H \cdot G$, where $G$ does not vanish in a neighborhood of the origin and a distinguished polynomial $H$ is of the form $H = x_k^m  + A_1(x_1 , \cdots , x_{k-1} ) x_k^{m-1}+ \cdots +  A_{m-1}(x_1 , \cdots , x_{k-1} ) x_k^1  + A_m ( x_1 , \cdots , x_{k-1}  ) $ with real analytic functions $A_1, \cdots A_m$ in $x_1 , \cdots , x_{k-1} $.

\bigskip
We consider the case of $h_0 >0$ in  Lemma \ref{geod55} (ii). For $z \in  W$, we can consider the geodesic $\theta_z (t) $ perpendicular to $W$ at $z$ such that  $\theta_z (0) =z$ and $\theta_z (\tilde{a} - \tilde{b}, 0)  :=\{ \theta_z(t) \ | \   t \in  (\tilde{a} - \tilde{b}, 0)   \}  \subset U$. Set $q(t):= f( \theta_z (t))$.
 The functions $q(t)$ and $ q^{'}(t)= df( \theta_z^{'} (t)) $  are  real analytic. So, $q^{'}(t)\neq 0$ on an interval $(0, \tilde{e}]$ for a number $\tilde{e}$ depending on $z$, and
$\nabla f \neq 0$  on $\theta_z  ( 0,  \tilde{e}]$.
  Setting $\tau = f (\theta_z( \tilde{e})  )$,
let $V^z$ be the connected component  of $f^{-1} ((b, \tau)) $  or $f^{-1} ((\tau, b)) $  which contains the curve $ \theta_z((0, \tilde{e} ))$. Set $W^z :=f^{-1}(b) \cap \overline{V^z}$.

%\begin{lemma} \label{ge421}Under the hypothesis of  Lemma \ref{geod55} {\rm (ii)},the following holds for $\theta_z( 0,  \tilde{e})$, $V_z$ and $W^z$ in the above paragraph.The curve  $\theta_z( 0,  \tilde{e})$ is tangent to $\frac{\nabla f}{|\nabla f|}$.

%{\rm (ii)}   $W= W^z$ for any $z \in W$.

%{\rm (iii)} We may let $V^z$ be independent of $z \in W$. \end{lemma}

By Lemma  \ref{lcf1v}, $V^z$ is of the form $\eta (\Sigma_x^0 \times (\tilde{c}, \tilde{d}) )$ for some numbers $x,  \tilde{c} ,\tilde{d}  $ where $\eta$ is defined on $V^z$ by the $\frac{\nabla f}{|\nabla f|}$ flow, i.e. $\eta_q(0) =q$ for $q \in \Sigma_x^0$  and $\eta^{'}_q(t)  = \frac{\nabla f}{|\nabla f|}$, just like  $\theta$ on $U$. By Lemma \ref{geod55}, $W^z$ is an $(n-1)$-dimensional real analytic submanifold of $M$ and
the geodesic $\eta_{q_z} (t) $ for some $q_z \in \Sigma_x^0$ meets $W^z$ perpendicularly at $z$.
We have $z \in W^z \cap W $ and  $\nabla f =0$ on
$W^z $  from Lemma \ref{threesolbx} (ii).
If $T_zW \neq  T_zW^z$, i.e. if $W$ and $W^z$ meet transversally,  then $U$ has a point where $\nabla f = 0$, a contradiction.
 So,  $T_zW=  T_zW^z$. As $\theta_z^{'}(0)  \perp T_z W$, we have
 $\theta_z^{'}(0) \perp T_z W^z$ so that the two geodesics $\eta_{q_z}  $ and $\theta_{z}  $ have the same trace.
Then $\theta_z( 0,  \tilde{e})$ is tangent to $\frac{\nabla f}{|\nabla f|}$. We can write $\theta_z(t+ c_0) = \eta_{q_z} (t )$ for a constant $c_0$.
% ?????  As this holds for any $z \in W$,we get $\theta(B_z \times (0, \epsilon)) \subset V^z$for an $(n-1)$-dimensional small neighborhood $B_z$ of $z$ in $W$ and a small number $\epsilon >0$.This implies $W= W^z$ as $W$ is connected.   Wemay let $V^z$ be independent of $z \in W$ due to Lemma \ref{threesolbx} (ii); indeed, set $V^z =  \theta (\Sigma_c^0 \times (\tilde{b}  , \tilde{b} +\tilde{e} )) .$

\bigskip
The next example warns that  we should not quickly conclude $\overline{V^z}$ and  $\overline{U}$ meet along $W$. Here the dependence of $\tilde{e}$ on $z$  can be not continuous.

\smallskip
{\bf Example 3}
Consider the real analytic function $f(x, y)= y(y-x^2 )$ on $\mathbb{R}^2$. Let $U = \{ (x, y) \ | \ y<0  \} $,  $V^{(0, 0)} = \{ (x, y) \ | \ y> x^2  \} $ and $W= \{ (x, y) \ | \ y=0  \}  $. Then $\overline{U}$ and  $\overline{V^{(0, 0)}}$ meet tangentially at the origin $(0,0)$, but not along $W$.

%\bigskip
%By Lemma \ref{ge421} (iii), we may set $V:= V^z$.

\begin{lemma} \label{geod55c} Under the hypothesis of Lemma \ref{geod55} {\rm (ii)}, the following {\rm (i)}$\sim${\rm (iii)} holds.

\smallskip
{\rm  (i)}   If  $U \cap V^z  \neq \emptyset$ for some $z \in W$, then $V^z \subset U$ or  $V^z \supset U$. Moreover, $ f^{-1}(b) \cap \overline{V^z}=W$ and  $\theta(  \Sigma_c^0 \times (\tilde{b}, \tilde{b} + \tilde{e} )  ) = V^z $.
In particular, $V^z$ can be independent of $z$.

\smallskip
{\rm (ii)} If $U \cap V^z  = \emptyset$ for all $z \in W$, then $W=W^z$ for any $z \in W$, and the map $\theta :  \Sigma_c^0 \times (\tilde{a}, \tilde{b} + \tilde{e} )   \rightarrow  M$  is a diffeomorphism onto $U \cup W \cup V^z$.
We may set $ V^z=\theta(  \Sigma_c^0 \times (\tilde{b}, \tilde{b} + \tilde{e} )  ) $ for any $z$
so that  $V^z$ is independent of $z$.

\smallskip
{\rm (iii)} Setting $V:=V^z$ for $z \in W$,  the Riemannian metric expression $g= ds^2  +    h(s)^2 \tilde{g}$  on $U$ extends onto   $U \cup W \cup V$.

\end{lemma}

\begin{proof}
Suppose  $U \cap V^z  \neq \emptyset$ for some $z \in W$.  From the definition of $U$ and $V^z$ and Lemma  \ref{lcf1v} (ii), and also by considering $f$-values on $U$ and $V^z$, it holds that
 $V^z \subset U$ or  $V^z \supset U$.  Clearly
$ f^{-1}(b) \cap \overline{V^z}=f^{-1}(b) \cap   \overline{U}=W$. The $\frac{\nabla f}{|\nabla f|}$ flow on $V^z$ meets $W$ perpendicularly by  Lemma \ref{geod55} (ii), so  the geodesic curve
 $\theta_z ( 0,  \tilde{e})$ is tangent to $\frac{\nabla f}{|\nabla f|}$ in $V^z$, and  $\theta(  \Sigma_c^0 \times (\tilde{b}, \tilde{b} + \tilde{e} )  ) = V^z $.   (i) is proved.

Suppose $U \cap V^z  = \emptyset$ for all $z \in W$.
 If there exists  $z \in W $ with $W = W^z$, then
 the map $\theta :  \Sigma_c^0 \times (\tilde{a}, \tilde{b} + \tilde{e} )   \rightarrow  U \cup W \cup V^z$ is a diffeomorphism.  We get $W=W^z$ for any $z \in W$. We may set $ V^z=\theta(  \Sigma_c^0 \times (\tilde{b}, \tilde{b} + \tilde{e} )  ) $ for any $z$.

Suppose that $W \neq W^{z}$  for each $z \in W$.  $W^z$ is a real analytic $(n-1)$-dimensional submanifold containing $z$, but  not contained in $\overline{U}$.
 By the Weierstrass Preparation theorem applied to the function $f-b$ near  $z_0 \in W$,
the zero set of $f-b$ in a ball neighborhood $B$ of $z_0$  is the zero set of its distinguished polynomial  $H(x_2 , \cdots , x_{n} ; s)  = s^m  + A_1 s^{m-1}+ \cdots +  A_{m-1}  s^1  + A_m $ where $A_1, \cdots A_m$ are real analytic functions of $x_2 , \cdots , x_{n}$ only.
So, there is some point $z_1$ in  $B\cap W$ and its small neighborhood $ B_{z_1} \subset B$ where  the zero set of $f-b |_{B_{z_1}} $ is  $B_{z_1} \cap W$. As $W^{z_1} \subset f^{-1}(b)$,
we get $W^{z_1} \cap B^{z_1} =W \cap B^{z_1} $. This implies $W^{z_1}=W$,  because they are both real analytic hypersurfaces in $f^{-1}(b)$ and the values of all derivatives of $f$ at a point determine $f$. This is a contradiction. Now (ii) is proved.
%Suppose $W = \{  (x, 0) \ | \ x \in \mathbb{R} \} \subset \mathbb{R}^2$  and $W^{z_1} \cap W=   \{  (x, 0) \ | \ x  \leq 0 \}$ and  $W \neq W^{z_1}$. I may consider a Weierstrass distinguished polynomial $H$ of $f$ near $(0, 0)$. We get  $H(x, y)  = y H_1 (x, y)$ for a real analytic function $H_1$.   $H_1(x, y) $ should vanish  on the closure of $W^{z_1}\setminus W$ which includes $(0, 0)$. Then the derivatives  $\frac{\partial^k H_1}{\partial {x}^k}$, $k=0, 1, \cdots$,  at $(0, 0)$ vanish. So,  $H_1(x, y) =y H_2(x, y) $. $H(x, y)  = y^2 H_2(x, y)$. This goes on so that $H(x, y)  = y^m H_m(x)$ for a real analytic function $H_m$. So, $H$ cannot vanish on $ W \cup W^{z_1} $. This is a contradiction.

%I may use the continuity of the roots of a polynomial in terms of its coefficients.
%???As the zero set of $f-b |_{B_{z_1}}$ contains  $\cup_{z \in W \cap  B_{z_1}} W^z \cap  B_{z_1} $, this is a contradiction to the description of $W^z$ in the previous paragraph.

%Note that in (i) and (ii), $V^z$ is independent of $z$.

The Riemannian metric expression   $g= ds^2 +    h(s)^2 \tilde{g}$  on $U$ extends smoothly onto   $U \cup W \cup V$  when $  U \cap V \neq  \emptyset $ and  when  $U \cap V =  \emptyset$.    (iii) is proved.
\end{proof}

 We remark that the Lemmas  \ref{lcf1v}$\sim$\ref{geod55c}  can be regarded as supplementing some argument in the proof of Theorem 1 and 2 of \cite{FG3}.

In  Lemma \ref{geod55c} (i), we have $U\subset V$ or $V\subset U$. By reducing $V$ or $U$ we can have $U=V$.

\begin{lemma} \label{da1f} In the case of  Lemma {\rm \ref{geod55c} (i)}, if $\tilde{b} -  \tilde{a}  = \tilde{e}$,
then $U=V$ and  the map $\theta : \Sigma_c^0 \times ( \tilde{a}, \tilde{b}+ \tilde{e} ) \longrightarrow  U \cup W \cup V$ is a $2:1$ covering map. There is a fixed-point free real analytic  map $\phi :  \Sigma_c^0  \longrightarrow  \Sigma_c^0 $ with $\phi \circ \phi =$Identity map so that  the map $\Phi: \Sigma_c^0 \times ( \tilde{a}, \tilde{b}+ \tilde{e} ) \longrightarrow  \Sigma_c^0 \times ( \tilde{a}, \tilde{b}+ \tilde{e} )$ defined by $\Phi (( p, t))   =(\phi(p),2 \tilde{b} -t      ) $ is a deck transformation.
In particular,
the region $U \cup W \cup V$ is  diffeomorphic  to the $\mathbb{Z}_2$-quotient manifold $ \Sigma_c^0 \times ( \tilde{a}, \tilde{b}+ \tilde{e} )/  \Phi$.
\end{lemma}
\begin{proof}
As $U \cap V \neq  \emptyset$ and $  \tilde{b} - \tilde{a} = \tilde{e} $, we get
$U=V$ by considering $f$-values.  Set $\tilde{\tau} = \tilde{b}+ \tilde{e} $.
As $U=V$, for $p \in \Sigma_c^0$,  $U \supset \theta_p ( ( \tilde{b}, \tilde{\tau} )  )$ which is perpendicular to $W$ at $z$.
So,  considering $f$-values, there are  points $q \neq p$ in $\Sigma_c^0$ such that
 $ \theta_q (( \tilde{a}, \tilde{b} ))$ in $U$ equals  $ \theta_p ( ( \tilde{b}, \tilde{\tau} )  ) $ in $V$. One can readily check that  $q$ is unique and  $\phi \circ \phi (p)  = p$ when we write $q = \phi(p)$. So, we have the identification $\theta (p, t)= \theta(\phi(p), 2\tilde{b} -t      ) $ for $ t  \in ( \tilde{a}, \tilde{\tau} )$.
  This bijective map $\phi :   \Sigma_c^0  \longrightarrow   \Sigma_c^0$ is real analytic  because a small neighborhood of $p$ in $ \Sigma_c^0 $ will be mapped via  $\phi$ onto its exact copy in $\Sigma_c^0$  through the above geodesics $\theta_p(t)$.
The real analytic diffeomorphism  $\Phi: \Sigma_c^0 \times ( \tilde{a}, \tilde{\tau} ) \longrightarrow  \Sigma_c^0 \times ( \tilde{a}, \tilde{\tau} )$ defined by $\Phi (( p, t))   =(\phi(p),2 \tilde{b} -t      ) $  is fixed-point free  and $\Phi^2 =${\rm Identity}. Moreover,  $\theta \circ \Phi  = \theta$ so that   $\Phi $ is a deck transformation for the $2:1$ covering map $\theta : \Sigma_c^0 \times ( \tilde{a}, \tilde{\tau} ) \longrightarrow  U \cup W \cup V$.
It follows that  $U \cup W \cup V$ is  diffeomorphic  to the $\mathbb{Z}_2$-quotient manifold $ \Sigma_c^0 \times ( \tilde{a}, \tilde{\tau} )/  \Phi$.
\end{proof}

We are ready to give a global description of $(M, g)$.

\begin{lemma} \label{ca5}
Let  $(M, g, f)$ be an $n$-dimensional connected complete  gradient Ricci soliton  with harmonic Weyl curvature. Assume that $M$ contains a type-{\rm (iii)} region of Theorem {\rm \ref{locals}}.
We then have only the following cases;

\medskip
{\rm (A)}   $(M, g)$ is isometric to $(\mathbb{S}^n, g_0)$ or its quotient, where $g_0 = ds^2 + h^2(s) \tilde{g}$ is a warped product metric on the sphere $\mathbb{S}^n$.  Here $\tilde{g}$ has positive constant curvature and $(M, g)$  is   locally conformally flat.

\smallskip
{\rm (B)}   $(M, g)$ is isometric  to  $(\mathbb{R}^n, g_0) $ where
$g_0= ds^2 + h^2(s) \tilde{g}$, $s \geq 0$ and $h(0)=0$.   Here $\tilde{g}$ has positive constant curvature and $(M, g)$  is   locally conformally flat.

%\smallskip {\rm (C)}  $(M, g)$ is isometric to  $(\Sigma_c^0 \times [a, b] / \phi , \  g_1) $ where $\Sigma_c^0$ is a connected component of $f^{-1} (c)$ for a regular value $c$ of $f$, $[a, b]$ is a closed interval, the boundry-identification map  $\phi: \Sigma_c^0 \times \{a\}   \longrightarrow  \Sigma_c^0 \times \{b\}$ is an isometry, $g_1=ds^2 + h^2(s) \tilde{g}$ with $h>0$ on $[a, b]$ and $\tilde{g}$ is  an Einstein metric on $\Sigma_c^0$.

\smallskip
{\rm (C)}  $(M, g)$ is isometric to   $(\Sigma_c^0 \times \mathbb{R},  g_0)  $  or its quotient,  where $\Sigma_c^0$ is a connected component of $f^{-1} (c)$ for a regular value $c$ of $f$,   and $g_0 = ds^2 + h^2(s) \tilde{g}$ with  $h>0$ on $\mathbb{R}$ and  $\tilde{g}$ is  an Einstein metric on $\Sigma_c^0$.
\end{lemma}
\begin{proof}

We may start with a connected component  $U$ of $f^{-1} (a, b)$, for some $a< b$, of the form $\theta (\Sigma_c^0 \times (\tilde{a}, \tilde{b}) )$ where $\nabla f \neq 0$,  as in  Lemma  \ref{lcf1v}.
As long as  $\nabla f \neq 0$ on  $f^{-1}(b) \cap   \overline{U}$ or  $f^{-1}(a) \cap   \overline{U}$, we may extend $U$ and get  a maximal connected component of the same kind.  We still denote it by  $U=\theta (\Sigma_c^0 \times (\tilde{a}, \tilde{b}) )$.
Clearly $\overline{U} \setminus U $ is the union of
$f^{-1}(b) \cap \overline{U}$ and  $f^{-1}(a) \cap \overline{U}$ if nonempty.

 If $\tilde{b} = \infty$, then $g = ds^2 + h^2(s) \tilde{g}$ on $U=\theta (\Sigma_c^0 \times (\tilde{a}, \infty ) )$ by  Lemma \ref{lcf1v}, so  that  $U$ stretches to the infinity as $s \rightarrow \infty$.

Suppose $\tilde{b} < \infty$. Then
$\nabla f =0$ along $W=f^{-1}(b) \cap   \overline{U}$. By  Lemma \ref{geod55},
if $\lim_{s \rightarrow  \tilde{ b}- } h(s)=0$, then the metric $g$ closes up to a point so that $\tilde{g}$ is a spherical metric and $W$ is a point, say $ \{p_0\}$, and $g$ is a smooth rotationally symmetric metric on  $U \cup \{ p_0 \}$.

If $\lim_{s \rightarrow \tilde{ b}- } h(s) >0$, by Lemma \ref{geod55c} and Lemma \ref{da1f} we get a maximal region $V$ so that
either $U \cup W \cup V$ is diffeomorphic to  $\Sigma_c^0 \times J$ via $ \theta$ for an open interval $J$, or
  $U \cup W \cup V$ is diffeomorphic to a $\mathbb{Z}_2$-quotient manifold of  $ \Sigma_c^0 \times J $.
In any case, by Lemma \ref{geod55c} (iii), $g$ can be written as  a warped product metric  $ ds^2 +    h(s)^2 \tilde{g}$ on the extended region.

The argument for $\tilde{b}$ in the above works similarly for $\tilde{a}$ too.
Now we repeat the above region-extending.
 Let us denote a region obtained at a step still by $U$. If $U$ has  two boundary components which coincide, $(M, g)$ belongs to  the case (C). Otherwise, $U$ may close up to a point at a boundary component,  may extend further as in  Lemma \ref{geod55c} (ii), or may need  Lemma \ref{geod55c} (i).

\medskip

 Then it is not hard to get eventually (A), (B) and  (C) only.
For instance, if $U$ has two boundary components $W_1$ and $W_2$ each of which is as in   Lemma \ref{da1f}, then
a covering manifold of $M$ can be diffeomorphic to  $\Sigma_c^0 \times \mathbb{R}$,  and
 $(M, g)$ belongs to the case (C). For another instances,  if  each of $W_1$ and  $W_2$  is a point (closing up), then $(M, g)$ belongs to  (A), and
 if  $W_1$ is a point and $W_2$ is as in Lemma \ref{da1f}, then $(M, g)$ still belongs to  (A). This proves the lemma.

\end{proof}

Now we prove Theorem 2 for steady solitons.

\medskip

\noindent {\bf  Proof of Theorem   \ref{steady}:}
Let  $(M, g, f)$ be an $n$-dimensional connected complete steady gradient Ricci soliton  with harmonic Weyl curvature.
Theorem \ref{locals} shows
 the possible four types (i)$\sim$(iv) of local regions of $M$. If $M$ admits a type-(i) region, then the soliton function $f$ is constant on the whole $M$ and $g$ is Ricci flat.
Next, $M$ cannot have a type-{\rm (ii)} region since $\lambda =0$.
 A type-(iv) region is not possible by  Lemma \ref{19}. We may assume that $M$ has only type-(iii) regions. We can use Lemma \ref{ca5}.
For  the locally conformally flat cases of (A) and (B),  $(M, g, f)$  is either flat or
 isometric to the Bryant soliton by \cite{CC1, CM}.

In the case of (C), a Riemannian covering $({\hat{M}}, \hat{g}, {\hat{f}})$ of $(M, g, f)$ is isometric to
 $( \mathbb{R} \times W^{n-1}, \  ds^2 +    \hat{h}^2 \tilde{g}  ) $, where  ${\hat{f}} = {\hat{f}}(s)$ and $\hat{h}=\hat{h}(s)>0 $ on  $ \mathbb{R} $, and  $\tilde{g}$ is an Einstein metric on a manifold $W^{n-1}$, say  $r_{\tilde{g} }= k (n-2) \tilde{g} $ for a constant $k$.
  As $({\hat{M}}, {\hat{g}}, {\hat{f}})$ is a  steady gradient soliton, $\nabla^{\hat{g}} d{\hat{f}} + r_{\hat{g}}=0$
where $\nabla^{\hat{g}}$ is the Levi-Civita connection of ${\hat{g}}$.

%%%%%  As $({\hat{M}}, \hat{g})$ is complete, we see that  $ \tilde{g}$ is  complete.   Let $\{ x_n\}$ be a Cauchy sequence in $( \Sigma_c^0 \times \{ 0\}, g|_{\Sigma_c^0 \times \{ 0\} } ) )$. Then $\{ x_n\}$ is a Cauchy sequence in $M$, which converges to $x_0$ in $M$. As $\Sigma_c^0 \times \{ 0\}$ is imbedded,   $\{ x_n\}$ converges to $x_0$ in   $( \Sigma_c^0 \times \{ 0\}, g|_{\Sigma_c^0 \times \{ 0\} } ) )$. So, $ \tilde{g}$ is  complete.

Let $\tilde{g}_k$ be a complete Riemannian metric of constant curvature on a new manifold $W_k^{n-1}$ with $r_{\tilde{g}_k }= k(n-2) \tilde{g}_k $.
Then, the new metric $G:=  ds^2 +    \hat{h}(s)^2 \tilde{g}_k$ on $ N :=\mathbb{R} \times W_k^{n-1}$ is a complete  locally conformally flat metric.  For $i \geq 2$, set  $E_i = \frac{1}{ \hat{h}} e_i$ where $ \{e_i\}$ is a local orthonormal frame field of $\tilde{g}_k$.
The Ricci tensor components of $ G$  are $ r_{   G}(\nabla s,\nabla s) = - (n-1) \frac{\hat{h}^{''}}{\hat{h}}$, $ r_{   G}(E_i,E_j) = \{ \frac{ (n-2)k}{\hat{h}^2} -  \frac{\hat{h}^{''}}{\hat{h}}  -  (n-2)\frac{(\hat{h}^{'})^2}{\hat{h}^2} \}  \delta_{ij}$  for $i, j >1$ and $ r_{   G}(\nabla s,E_i) =0$. Considering $\hat{f}(s) $ as defined on $N$, we get
$\nabla^{G} d{\hat{f}} (\nabla s,  \nabla s)= {\hat{f}}^{''}$, $\nabla^{G} d{\hat{f}} (E_i,  E_j)= \frac{\hat{h}^{'}}{\hat{h}} {\hat{f}}^{'} \delta_{ij} $ and  $\nabla^{G} d{\hat{f}} (\nabla s,  E_i)=0 $. We can get the corresponding same formulas for
$r_{\hat{g}}$ and $\nabla^{\hat{g}} d{\hat{f}}$. Since  $( {\hat{M}}, {\hat{g}}, {\hat{f}} )$ is a steady gradient Ricci soliton,  so is  $( N, G, {\hat{f}} )$.
Now by \cite{CC1, CM}, $( N, G,{\hat{f}} )$ is either flat or
 isometric to the Bryant soliton.

If $( N, G, {\hat{f}} )$ is  flat, then we get $\hat{h}^{''}=0$ and $\hat{h}= as +b $, for constants $a$ and $b$.
If $a \neq 0$, we have a contradiction to $\hat{h}(s)>0$ on $\mathbb{R}$.
 If $a=0$, then $\hat{h}=b>0$ and $({\hat{M}}, {\hat{g}}, {\hat{f}})$ is the Riemannian product
 $( \mathbb{R} \times W^{n-1}, \  ds^2 +    b^2 \tilde{g}  ) $, where  $\tilde{g}$ is a Ricci flat metric on a manifold $W^{n-1}$.
 So,  $(M, g, f)$ is  Ricci flat.

 %As $F^{''} =0$,  $F(s)$ is a linear function of $s$. If $F$ is a constant,  If $F$ is   a non-constant linear function of $s$,

 If $( N, G, \hat{f} )$ is  isometric to the Bryant soliton, the scalar curvature $R_G$ of $G=  ds^2 +    \hat{h}(s)^2 \tilde{g}_k$ is a non-constant function of $s$ only. So, the level hypersurfaces of  $R_G$  does not close up to a point because  $\hat{h}(s)>0 $ on  $ \mathbb{R} $.  But the level hypersurfaces of  the scalar curvature  in the Bryant soliton  close up to a point \cite[Chap. 1,  Section 4]{Chow}, a contradiction.
In summary, we have proved Theorem \ref{steady}.

 % using the Ricci-eigen vector field $E_1 = \frac{  \nabla F }{| \nabla F |} = \frac{d}{ds}$, the soliton equation $\nabla^{G} dF + r_{G}=0$ becomes \begin{eqnarray}F^{''} - (n-1) \frac{h^{''}}{h}    =   0,  \hspace{2.5cm} \label{fh881a} \\ \frac{h^{'}}{h} F^{'}  + \frac{(n-2)k}{h^2} -  \frac{h^{''}}{h}  -  (n-2)\frac{(h^{'})^2}{h^2} = 0.\label{fh882a}\end{eqnarray}The Bryant soliton is rotationally symmetric and $h$ closes off at some point.

\bigskip

Next we prove Theorem  \ref{expand} for expanding solitons.

\smallskip
\noindent {\bf  Proof of Theorem   \ref{expand}:}
Let  $(M, g, f)$ be an $n$-dimensional connected complete expanding gradient Ricci soliton  with harmonic Weyl curvature.
As $\lambda<0$, $M$ can have regions
 of type (i)$\sim$(iii) by Theorem \ref{locals}. If $M$ admits a type-(i) region, then the soliton function $f$ is constant on the whole $M$ and $g$ is Einstein.
%By the paragraph before Lemma \ref{lcf1v}, a type-(ii) region and a type-(iii) region cannot be contained in $M$ simultaneously.
If $(M, g)$ admits a type-(ii) region, then  the scalar curvature is constant on $M$ and we  get the case (ii) of  Theorem  \ref{expand} from \cite{FG2, PW}.

It is well known that any compact expanding Ricci soliton is Einstein with constant $f$.
So, if $(M, g)$ admits a type-(iii) region, by Lemma \ref{ca5} we get the cases of
{\rm (B)} and
{\rm (C)}.
From these we can get (iii) and (iv) of  Theorem  \ref{expand}.  We have finished the proof.

\begin{remark}
{\rm Some argument in classifying gradient Ricci solitons with harmonic Weyl curvature can be applied to studying other geometric spaces.  Some more results in this direction will appear elsewhere. }
\end{remark}

\bigskip \noindent
\footnotesize{ Jongsu Kim: Dept. of Math., Sogang University, Seoul, Korea; \ \ jskim@sogang.ac.kr}
\end{document}